\documentclass[reqno]{amsart}

\usepackage[utf8]{inputenc}
\usepackage[T1]{fontenc}
\usepackage[english]{babel}
\usepackage{microtype}
\usepackage[colorlinks=true,allcolors=black]{hyperref}
\usepackage{lmodern}
\usepackage{amsmath, amssymb, amsthm}
\usepackage{mathtools}
\usepackage{mathrsfs}
\usepackage{bbm}
\usepackage{cite}
\usepackage{enumitem}
\usepackage{amsthm}
\usepackage{multirow}
\usepackage{algorithm}
\usepackage{algpseudocode}
\usepackage{tabularx}
\usepackage{graphicx}
\usepackage{caption}
\usepackage{url}
\usepackage{amsaddr}
\usepackage{afterpage}
\usepackage[draft]{changes}
\definechangesauthor[name={KB}]{KB}
\definechangesauthor[name={SB}]{SB}

\setlist[enumerate]{label=\arabic*.}

\numberwithin{equation}{section}

\addto\captionsenglish{}

\newcommand{\R}{\mathbb{R}}
\newcommand{\N}{\mathbb{N}}

\newcommand{\bS}{\mathbb{S}}

\newcommand{\HH}{\mathcal{H}}

\newcommand{\LL}{\mathcal{L}}

\newcommand{\dd}{\mathrm{d}}
\newcommand{\D}{\mathrm{D}}

\newcommand{\SBV}{\mathrm{SBV}}
\newcommand{\BV}{\mathrm{BV}}

\newcommand{\GSBV}{\mathrm{GSBV}}

\newcommand{\AT}{\mathcal{AT}}
\newcommand{\MS}{\mathcal{MS}}

\newcommand{\1}{\mathbbm{1}}
\newcommand{\G}{\mathcal{G}}

\newcommand{\Om}{\Omega}
\newcommand{\eps}{\varepsilon}
\newcommand{\wto}{\rightharpoonup}
\newcommand{\half}{\frac{1}{2}}
\newcommand{\tforall}{\text{for all }}
\newcommand{\tfor}{\text{for }}
\newcommand{\loc}{\text{loc}}

\newcommand{\coloneq}{\mathrel{\mathop:}=}

\newcommand{\abs}[1]{\lvert #1 \rvert}
\newcommand{\bigabs}[1]{\bigl\lvert #1 \bigr\rvert}

\newcommand{\norm}[1]{\lVert #1 \rVert}
\newcommand{\bignorm}[1]{\bigl\lVert #1 \bigr\rVert}
\newcommand{\Bignorm}[1]{\Bigl\lVert #1 \Bigr\rVert}
\newcommand{\biggnorm}[1]{\biggl\lVert #1 \biggr\rVert}

\newcommand{\scprod}[2]{\langle #1 , #2 \rangle}
\newcommand{\bigscprod}[2]{\bigl\langle #1 , #2 \bigr\rangle}

\newcommand{\biggscprod}[2]{\biggl\langle #1 , #2 \biggr\rangle}

\DeclareMathOperator*{\argmin}{arg\,min}
\DeclareMathOperator{\dist}{dist}

\DeclareMathOperator{\supp}{supp}
\DeclareMathOperator*{\Glim}{\Gamma-lim}
\DeclareMathOperator*{\Gliminf}{\Gamma-lim\,inf}
\DeclareMathOperator*{\Glimsup}{\Gamma-lim\,sup}

\DeclareMathOperator*{\essinf}{ess\,inf}
\DeclareMathOperator*{\esssup}{ess\,sup}

\DeclareMathOperator{\diver}{div}
\DeclareMathOperator{\prox}{prox}

\urldef{\alinaseye}\url{https://www.flickr.com/photos/angel_ina/3201337190/}
\urldef{\sailing}\url{https://www.flickr.com/photos/124790945@N06/32397550408/}
\urldef{\license}\url{https://creativecommons.org/licenses/by/2.0/}

\theoremstyle{plain}
\newtheorem{theorem}{Theorem}[section]
\newtheorem{lemma}[theorem]{Lemma}
\newtheorem{corollary}[theorem]{Corollary}
\newtheorem{proposition}[theorem]{Proposition}

\theoremstyle{definition}

\newtheorem{assumption}[theorem]{Assumption}

\theoremstyle{remark}
\newtheorem{remark}[theorem]{Remark}

\usepackage{xcolor}

\title[$\BV$-phase field approximation for the Mumford-Shah functional]{Approximation of the Mumford-Shah Functional by Phase Fields of Bounded Variation}

\author[S. Belz]{Sandro Belz}
\address{Department of Mathematics, Technical University of Munich\\Boltzmannstraße 3, 85748 Garching, Germany\\sandro.belz@tum.de}

\author[K. Bredies]{Kristian Bredies}
\address{Institute of Mathematics and Scientific Computing, University of Graz\\Heinrichstraße 36, 8010 Graz, Austria\\kristian.bredies@uni-graz.at}

\allowdisplaybreaks

\begin{document}

\begin{abstract}
  In this paper we introduce a new phase field approximation of the Mumford-Shah functional similar to the well-known one from Ambrosio and Tortorelli. However, in our setting the phase field is allowed to be a function of bounded variation, instead of an $H^1$-function. In the context of image segmentation, we also show how this new approximation can be used for numerical computations, which contains a total variation minimization of the phase field variable, as it appears in many problems of image processing. A comparison to the classical Ambrosio-Tortorelli approximation, where the phase field is an $H^1$-function, shows that the new model leads to sharper phase fields.
\end{abstract}



\maketitle

\section{Introduction}

The Mumford-Shah functional has been introduced by D. Mumford and J. Shah in \cite{MumSha1989} in the context of image segmentation. For a given image, $g \in L^\infty(\Om)$, where $\Om \subset \R^n$ represents the image domain, it is given by
\begin{equation}
  \label{strongMS}
  \frac{\alpha}{2} \int_\Omega \abs{\nabla u}^2 \, \dd x + \frac{\beta}{2} \int_{\Omega} \abs{u-g}^2 \,\dd x + \gamma \HH^1 (\Gamma)
\end{equation}
where $\alpha, \beta, \gamma >0$ are parameters, free to choose. One wants to minimize the functional with respect to $u \in C^1(\Omega \setminus \Gamma)$, being the segmentally denoised approximation of $g$, and $\Gamma \subset \Om$ closed, describing the contours of the segments.
For $\beta = 0$ this functional appeared once more in \cite{FraMar1998}  in the context of fracture mechanics. There, $u$ models the displacement function, and $\Gamma  \subset \Om$ being closed represents the fracture set. The minimization is then restricted to some Dirichlet boundary condition.

As usual in the theory of free-discontinuity problems (see \cite{AmbFusPal2000,Bra1998}) the Mumford-Shah functional \eqref{strongMS} is relaxed to the space of special functions of bounded variation (see Section~\ref{functionsOfBoundedVariation} for more details on these functions), where the set $\Gamma$ is replaced by the discontinuity set $S_u$. Namely, instead of \eqref{strongMS} one considers
\begin{equation}
  \label{SBV_MS}
  \frac{\alpha}{2} \int_\Omega \abs{\nabla u}^2 \, \dd x + \frac{\beta}{2} \int_{\Omega} \abs{u-g}^2 \,\dd x + \gamma \HH^1 (S_u)
\end{equation}
for $u\in \SBV(\Om)$, the set of special functions of bounded variation.
In this setting the existence of the minimizers is well-known  and follows from compactness properties of $\SBV(\Om) \cap L^\infty(\Om)$ and some lower semi-continuity properties (see \cite{Amb1995,Amb1989,Amb1990}), using the direct method in the calculus of variations.
Furthermore, by the regularity property shown in \cite{DeCarLea1989} we know that for any minimizer  $u \in \SBV(\Om)$ of %
\eqref{SBV_MS}
the pair  $(u,\bar S_u)$ minimizes~\eqref{strongMS}.

As already mentioned, in the case of fracture mechanics, we usually have $\beta = 0$ and $L^2$\nobreakdash-penalization is replaced by a Dirichlet boundary condition. In general, the functional must then be defined on $\GSBV(\Om)$, the set of generalized special functions of bounded variation (see Section~\ref{functionsOfBoundedVariation}), in order to obtain the existence of a minimizer. This is due to the requirement of a uniform bound of the minimizing sequence in the direct method for applying the above-mentioned compactness properties in $\SBV(\Om)$. Only for $\beta >0$, this $L^\infty$-bound is automatically achieved, whereas for $\beta = 0$ one has to fall back to $\GSBV(\Om)$.

For numerical computations some variational approximations in terms of $\Gamma$-convergence (see Section~\ref{GammaConvergence}) turned out to be very useful.  It guarantees that a convergent sequence of minimizers of the approximating functionals converge to minimizers of the $\Gamma$\nobreakdash-limit.
We firstly discuss in Theorem~\ref{mainTheorem} an approximation for $\beta=0$. Hence, we consider the functional
\begin{equation*}
  \MS(u) \coloneq  \frac{\alpha}{2}  \int_\Omega \abs{\nabla u}^2 \, \dd x + \gamma \HH^1 (S_u) \quad \tforall u\in \GSBV(\Om) \,.
\end{equation*}

One of the first and most popular results in this direction was given by L. Ambrosio and V. M. Tortorelli in \cite{AmbTor1992}. They introduced the functionals
\begin{equation}
  \label{AmbrosioTortorelli}
  \mathcal{AT}_\eps (u,v) = \int_\Omega (v^2 + \eta_\eps) \abs{\nabla u}^2 \, \dd x + \int_\Omega \frac{1}{4 \eps}  (1-v)^2 + \eps  \abs{\nabla v}^2 \,\dd x
\end{equation}
for $u \in H^1(\Om)$ and $v\in H^{1} (\Om; [0,1])$ and showed via a $\Gamma$-convergence argument that any limit point $(u,1)$ of a sequence of minimizers $(u_ \eps , v_\eps)$ of $\AT_\eps$ is a minimizer of $\MS$, provided that $\frac{\eta_\eps}{\eps} \to 0$.
Many other approximations based on this result have been proven. Just recently, we proved that the Euclidean norms of the gradients can be replaced by  Riemannian norms (see \cite{AlmBelMicPer2018}). This result finds application in fracture mechanics applied to surfaces. Another approach considering higher order terms of the phase field has been studied e.g. in \cite{BorHugLanVer2014} and \cite{BurEspZep2015}.
What happens with the approximation $\AT_\eps$ when $\frac{\eta_\eps}{\eps}$ does not converge to zero is investigated in \cite{DalIur2013} and \cite{Iur2013}.
A totally different idea of approximating $\MS$ by finite differences was proposed by E. De Giorgi and proven by M. Gobbino in \cite{Gob1998}. In \cite{BraDal1997} A. Braides and G. Dal Maso used non-local functionals depending on the average of the gradient of $u$ on small balls.
From the work presented in \cite{Bra1998} one gets an approximation of $\MS$ for the following functional with small $\eps >0$:
\begin{equation}
  \label{ATp}
  \int_\Omega (v^2 + \eta_\eps) \abs{\nabla u}^2 \, \dd x +  \frac{1}{2^{p'} \eps} \int_\Omega (1-v)^{p'} \,\dd x + \eps^{p-1} \int_\Omega \abs{\nabla v}^p \,\dd x
\end{equation}
for $u \in H^1(\Om)$ and $v \in W^{1,p}(\Om)$ with $p>1$ and $p'$ being the
Hölder conjugate of $p$.

In the approximations \eqref{AmbrosioTortorelli}, \eqref{ATp} and in the functionals we are going to study in this paper the additional function $v$ works as a phase field variable describing the discontinuity set of $u$. To be more precise, for small $\eps > 0$ the minimizing function $v$ is close to $0$ where $u$ is ``steep'' or jumps, which means in the context of fracture mechanics the presence of a crack and in the context of image segmentation the presence of a contour. Elsewhere, the phase field variable is close to 1 and $u$ is expected to be ``flat'' in this area. In practice the weights of the different integral terms declare what is meant to be ``steep'' or ``flat''.

In this paper we present a new approximation of the Mumford-Shah functional, allowing the phase field variable $v$ to be in $\BV(\Om)$, the set of functions of bounded variation. Namely, as a special case of our main result we consider the functionals
\begin{equation*}
  \frac{\alpha}{2} \int_\Om (v^2 + \eta_\eps)
    \abs{\nabla u}^2 \, \dd x + \frac{\gamma}{2 \eps} \int_\Om (1-v) \,\dd x +
    \frac{\gamma}{2} \abs{\D v}(\Om)
\end{equation*}
for $u \in H^1(\Om)$ and $v\in \BV(\Om)$, which $\Gamma$-converge in some sense to $\MS$ and represent the case with $p=1$ in \eqref{ATp}. From there we derive the required setting for $\beta >0$.

In this way the phase field variable $v$ can have jumps, which is exploited in the proof of the $\limsup$-inequality (see Proposition~\ref{limsupInequality}), when constructing the recovery sequence for the $\Gamma$\nobreakdash-convergence result. Moreover, we expect from this fact that the phase fields become somewhat sharper than the ones obtained from  \eqref{AmbrosioTortorelli}. We approve this expectation with some numerical computations in the context of image segmentation. The application of this theory to fracture mechanics remains for future work, which includes the studying the convergence behaviour of a time-discrete evolution as it was done -- and is still ongoing -- for the classical approach of Ambrosio and Tortorelli. For more details on this topic we refer to \cite{Gia2005,AlmBelNeg2018,AlmBel2018,Neg2014,Neg2019,KneNeg2017,AlmNeg2020}.

The paper is structured as follows: In Section~\ref{preliminaries} we start with some preliminaries recalling the necessary technical issues. For the versed reader this section might be skipped or only used as a reference text.  In Section~\ref{sec:main-result} we formulate our main result, Theorem~\ref{mainTheorem}, from which we directly infer all other necessary theorems and corollaries.  Section~\ref{proof} is then dedicated to the proof of Theorem~\ref{mainTheorem} and in Section~\ref{numerics} we provide some  numerical comparison of our new  model and the classical Ambrosio-Tortorelli approximation.

\section{Preliminaries and Notation}
\label{preliminaries}
In this section, we collect the notation and the well-known results from the literature which are used in this paper.

With $B_\rho (x)$ we denote the Euclidean ball with radius $\rho > 0$ and center $x \in \R^n$.  For some $S\subset \R^n$ the set $B_\rho(S)$ refers to the $\rho$\nobreakdash-neighborhood of $S$. The set $\bS^{n-1}$ is the $n-1$-dimensional sphere in $\R^n$. At some places it  is convenient to use the short notation $a \vee b$ and $a \wedge b$ for $\max\{a,b\}$  and $\min\{a, b\}$, respectively.

The essential supremum and the essential infimum of some measurable function~$u$ is written as $\esssup u$ and $\essinf u$, respectively. The essential support of a measurable function $u$ is denoted by $\supp u$.

\subsection{Measure theory}
For any set $\Om \subset \R^n$ we denote by $\LL^n (\Om)$ the $n$\nobreakdash-dimensional Lebesgue measure and by $\HH^k (\Om)$ the $k$\nobreakdash-dimensional Hausdorff measure. Instead of $\HH^0$ we also use the symbol $\#$ for the counting measure.  For a (signed, vector-valued) measure $\mu$ we write $\abs{\mu}$ for its total variation.

\subsection{\texorpdfstring{$\Gamma$-convergence}{Gamma-convergence}}
\label{GammaConvergence}
For some sequence of functionals $(F_j)$ and a functional $F$ defined on some metric space~$X$ we say that $F_j$   \emph{$\Gamma$\nobreakdash-converges} to $F$ as $j \to \infty$ and write $\Glim_{j \to \infty} F_j = F$ if there holds the
\begin{description}
\item[$\boldsymbol{\liminf}$-inequality]
  for all $u\in X$ and all sequences $(u_j)$ in $X$ with $u_j \to u$ there holds
  \begin{equation}\label{liminfInequality}
    F(u) \leq \liminf_{j\to\infty} F_j (u_j) \,.
  \end{equation}
\item[$\boldsymbol{\limsup}$-inequality]
  for all $u \in X$ there exists a sequence $(u_j)$ in $X$ such that $u_j \to u$ and
  \begin{equation}\label{limsupInequality}
    \limsup_{j\to \infty} F_j(u_j) \leq F(u) \,.
  \end{equation}
\end{description}
One often defines
\begin{align*}
  \Gliminf_{j \to \infty} F_j (u) &\coloneq \inf \{ \liminf_{j \to \infty} F_j (u_j) \colon u_j \in X \text{ for all } j > 0, u_j \to u \text{ as } j \to \infty \} \,, \\
  \Glimsup_{j \to \infty} F_j (u) &\coloneq \inf \{ \limsup_{j \to \infty} F_j (u_j) \colon u_j \in X \text{ for all } j > 0, u_j \to u \text{ as } j \to \infty \} \,.
\end{align*}
Then the $\liminf$-inequality is equivalent to $F \leq \Gliminf_{j \to \infty} F_j$ and the $\limsup$-inequality is equivalent to $\Glimsup_{j \to 0} F_j \leq F$. Note that~$\Gliminf_{j \to \infty} F_j$ as well as~$\Glimsup_{j \to \infty} F_j$ are lower semi-continuous.

If one has a family of functionals $(F_\eps)$ for $\eps \in I \subset \R$ the definition is adapted in the usual way, i.e. $F_\eps$ $\Gamma$-converges to $F$ as $\eps \to a$ (for some $a\in \bar I $) if $F_{\eps_j}$ $\Gamma$-converges to $F$ for all sequences $(\eps_j)$ in $I$ with $\eps_j \to a$.

The most important property of $\Gamma$\nobreakdash-convergent sequences is the convergence of minimizers to a minimizer of the limit functional, which is stated in the following proposition.
\begin{proposition}
	\label{prop:minconvergence}
  Let $(F_\eps)$, where $F_\eps \colon X \to \R \cup \{\infty \}$, be a sequence of functionals $\Gamma$\nobreakdash-converging to $F\colon X \to \R \cup \{\infty \}$ with respect to the metric space $X$. Assume that $\inf_X F_\eps = \inf_K F_\eps$  for some compact set $K\subset X$. Then, there holds $\lim_{\eps\to 0} \inf_X F_\eps = \inf_X F$. Furthermore, for any sequence $(u_\eps)$ in $X$ converging to $u \in X$ with $F_\eps (u_\eps)  =\inf_X F_\eps $ we have $F (u) = \inf_X F$.
\end{proposition}

If $F = \Glim_{j \to \infty} F_j$ and $u\in X$, a sequence $(u_j)$, for which \eqref{limsupInequality} holds, is called a \emph{recovery sequence} for $u$, and there clearly holds $\lim_{j \to \infty} F_j (u_j) = F(u)$.
It is actually the case that a convergent sequence of minimizers is a recovery sequence for the minimizer of the $\Gamma$-limit. For this reason  knowing the recovery sequences provides lots of information about the structure of the limit behaviour of the functional sequence.

For more details on the concept of $\Gamma$\nobreakdash-convergence we refer to \cite{Bra2002} and \cite{Dal1993}.

\subsection{Functions of bounded variation}
\label{functionsOfBoundedVariation}
In the following we describe the concept and some essential results of functions of bounded variation. For an extensive monograph on this topic we refer to \cite{AmbFusPal2000}. A more basic introduction can be found in \cite{EvaGar1992}.

Let $\Om \subset \R^n$ be non-empty and open for the rest of this section. The set of functions of bounded variation, in short $\BV(\Om)$, contains all functions $u\in L^1(\Om)$ whose distributional derivative is a Radon measure, denoted by $\D u$, i.e. there holds
\begin{equation}\label{BVpartialIntegration}
  \int_\Om u \diver w \,\dd x = - \int_\Om w \,\dd \D u \quad \tforall w \in C^1_c (\Om; \R^n) \,.
\end{equation}

Defining the total variation
\begin{equation}
  \label{variation}
  V(u,\Om) = \sup \biggl\{ \int_\Om u \diver w \, \dd x : w \in C_c^1(\Om; \R^n), \norm{w}_\infty \leq 1 \biggr\}
\end{equation}
we obtain from the Riesz representation theorem that \eqref{BVpartialIntegration} is equivalent to\linebreak $V(u, \Om) < \infty$. Furthermore, there holds $\abs{\D u} (\Om) = V(u, \Om)$ for all $u\in \BV( \Om)$.

For any measurable function $u \colon \Om \to \R$ we define for all $x \in \Om$ the upper and lower approximate limit, respectively, by
\begin{align*}
  u^+(x) &= \inf \biggl\{ t \in \R : \lim_{\rho \to 0}\frac{\LL^n \bigl( \{u > t\} \cap B_\rho (x) \bigr) }{\rho^{n}}  = 0\biggr\}\,, \\
  u^-(x) &= \sup \biggl\{ t \in \R : \lim_{\rho \to 0}\frac{\LL^n \bigl( \{u < t\} \cap B_\rho (x) \bigr) }{\rho^{n}} =0 \biggr\}\,.
\end{align*}
For all $x\in \Om$ there obviously holds $u^-(x) \leq u^+(x)$. If $u^- (x) = u^+(x)$ we write for their common value $u^\ast (x)$.  The set $S_u$ is the discontinuity set containing all those points $x\in \Om$ for which there holds $u^-(x) < u^+(x)$.

In what follows let $u\in \BV(\Om)$. Then, $S_u$ has Lebesgue measure 0 and for $\HH^{n-1}$-almost all points $x\in S_u$ one can find a unit normal vector $\nu_u(x)$ such that $u^+ (x)= \bigl(u\rvert_{H^+(x)}\bigr)^\ast (x)$ and $u^- (x)= \bigl( u\rvert_{H^-(x)} \bigr)^\ast (x)$ with
\begin{align*}
  H^+(x) &= \bigl\{y \in \Om \colon \bigscprod{y-x}{\nu_u(x)} > 0 \bigr\} \\
  H^-(x) &= \bigl\{y \in \Om \colon \bigscprod{y-x}{\nu_u(x)} < 0 \bigr\} \,.
\end{align*}
If this is the case one says that $x$ is a jump point.
We call $\tilde u$ a \emph{precise representative} of $u$ if  $\tilde u (x) = u^\ast (x)$ for all~$x \in \Om \setminus S_u$ and  $ \tilde u (x)= \half (u^+(x) + u^-(x))$ for all jump points $x \in S_u$.

For functions of bounded variation on the real line we actually have that every point in $S_u$ is a jump point. Furthermore, on an open interval the pointwise variation of $\tilde u$ and the variation as defined in \eqref{variation}  coincide. Precisely, for $ a < b$ and $u \in\BV(a,b)$ there holds
\begin{equation}
  \label{pointwiseVariation}
  V \bigl(u,(a,b)\bigr) = \sup \Biggl\{ \sum_{i=1}^N \bigabs{\tilde u (t_{i})  - \tilde u( t_{i-1})}  : N \in \N, a < t_0 < \dots < t_N < b \Biggr\} \,.
\end{equation}

For any $u\in \BV (\Om)$ one can split the measure $\D u$ in the following way
\begin{equation*}
  \D u = \D^a u + \D^j u + \D^c \,,
\end{equation*}
where  $\D^a u = \nabla u \LL^n$ denotes the absolutely continuous part of $\D u$ with respect to the Lebesgue measure. Therefore, with $\nabla u$ we denote its density function, which we also call the \emph{approximate gradient} of $u$.
With $\D^j u $ we denote the  \emph{jump part} of $u$, which can be written as $\D^j u= (u^+ - u^-) \cdot \nu_u \HH^{n-1} \llcorner S_u$, and $\D^c u$ is the Cantor part.

The set of special functions of bounded variation, denoted by $\SBV(\Om)$, contains those functions of bounded variation whose Cantor part is zero, i.e.~we have~$\SBV(\Om) = \{ u\in \BV(\Om) : \D^c u = 0 \}$. The singular part of such functions is therefore only concentrated on the set of jump points.

A measurable function $u\colon \Om \to \R$ is a generalized special function of bounded variation, where we write $u\in \GSBV(\Om)$, if any truncation of $u$ is locally a special function of bounded variation, i.e.  $u^M \in \SBV_\loc (\Om)$ for all $M > 0$, with $u^M = (-M) \vee u \wedge M$. Note that for $u\in \GSBV(\Om)$ we cannot define $\nabla u$ as above, because the distributional derivative does not need to be a measure on that space. However, $\nabla u^M$ is well defined for all $M>0$ and converges pointwise a.e.~as $M\to \infty$. Thus, we simply define $\nabla u (x) = \lim_{M\to \infty} \nabla u^M(x)$ for a.e.~$x\in \Om$. Furthermore, one can show that $S_u = \bigcup_{M>0} S_{u^M}$. These results and more details can be found in~\cite[Section~4.5]{AmbFusPal2000} and the references therein.

Moreover, we will use the following two subspaces of $\GSBV(\Om)$ and $\SBV(\Om)$ defined for every $p>0$ by
\begin{align*}
	\SBV^p (\Om) &= \bigl\{ u \in \SBV (\Om) \colon \nabla u \in L^p (\Om), \HH^{n-1} (S_u) < \infty \bigr\} \\
	\GSBV^p (\Om) &= \bigl\{ u \in \GSBV (\Om) \colon \nabla u \in L^p (\Om), \HH^{n-1} (S_u) < \infty  \bigr\} \,.
\end{align*}

A density result, which plays an important role in the proof of the $\limsup$-inequality for our main assertion, is stated in the next theorem. It follows directly from \cite[Theorem~3.1]{CorToa1999} and the following remarks therein.
\begin{theorem} \label{densityResult}
  Let $\Om \subset \R^n$ be non-empty, open and bounded with Lipschitz boundary, and take $u \in \SBV^2 (\Om) \cap L^\infty (\Om)$. Then, there exists a sequence  $(w_j)$ in $\SBV^2 (\Om) \cap L^\infty (\Om)$ such that
  \begin{enumerate}
  \item \label{firstDensityProp} $\displaystyle \overline{S_{w_j}} \text{ is a polyhedral set,}$
  \item $\displaystyle \HH^{n-1} \bigl(\overline{S_{w_j}} \setminus S_{w_j} \bigr) = 0 \,,$
  \item \label{middleDensityProp} $w_j \in W^{1,\infty} (\Om \setminus S_{w_j}) \text{ for all } j\in \N \,,$
  \item $\displaystyle w_j \to u \text{ in } L^1(\Om) \text{ as } j \to \infty \,,$
  \item \label{L2-convergence} $\displaystyle \nabla w_j  \to \nabla u \text{ in } L^2 (\Om) \text{ as } j \to \infty \,,$
  \item \label{lastDensityProp} $\displaystyle \HH^{n-1} (S_{w_j}) \to \HH^{n-1} (S_w) \text{ as } j \to \infty \,.$
  \end{enumerate}
\end{theorem}

We now shortly introduce the concept of slicing, which is essential for the proof of the $\liminf$-inequality. For that purpose, let $\Omega \subset \R^n$ be open and bounded, and let $\xi \in \bS^{n-1}$ be a unique normal vector. Then, we write $\Om_\xi$ for the projection of $\Om$ onto $\xi^\perp$, and we set
\[
  \Om^\xi_y \coloneq \{ t\in \R \colon y + t \xi \in \Omega \} \quad \tforall y\in \Om_\xi \,.
\]
Furthermore, for any function $u\in L^1 (\Om)$ and for $\LL^{n-1}$\nobreakdash{-}a.a. $y \in \Om_\xi$ we can define $u_y^\xi (t) \coloneq u(y + t \xi)$ for a.a. $t\in \Om^\xi_y$.

One can show the following important results revealing the connection between a function $u\in \SBV (\Om)$ and its sliced functions $u_y^\xi$. There are more general results for $\BV$-functions, which are not needed in this context. The interested reader can find the details in \cite[Section~3.11]{AmbFusPal2000}.
\begin{theorem}\label{slicingProp}
  Let $u\in L^1(\Om)$. Then $u\in \SBV(\Om)$ if and only if for all $\xi \in \bS^{n-1}$ there holds $u_y^\xi \in \SBV (\Om_y^\xi)$ for $\LL^{n-1}$\nobreakdash-a.a. $y\in \Om_{\xi}$ and
  \[
    \int_{ \Om_\xi } \bigabs{\D u_y^\xi} \bigl( \Om_y^\xi \bigr) \dd \LL^{n-1} (y) < \infty \,.
  \]
  Furthermore, if $u\in \BV(\Om)$ there holds for all $\xi \in \bS^{n-1}$, for $\LL^{n-1}$\nobreakdash{-}a.a. $y\in \Om_\xi$ and for a.a. $t\in \Om_y^\xi$:
  \begin{enumerate}
  \item
    $(u_y^\xi)' (t) = \bigscprod{\nabla u (y + t\xi)}{\xi} $,
  \item
    $S_{u_y^\xi} = (S_u)_y^\xi$,
  \item$
    (u_y^\xi)^{\pm} (t) = u^{\pm} (y + t\xi)$,
  \item
    $\bigabs{\scprod{\D^\ast u}{\xi}} (\Om) = \int_{\Om_{\xi}} \bigabs{\D^\ast u_y^\xi} (\Om^\xi_y) \,\dd \LL^{n-1}( y) \quad \text{ for } \ast = a,j,c$.
  \end{enumerate}
\end{theorem}

The following corollary directly follows by a truncation argument.
\begin{corollary} \label{slicingCor}
  Let $u\in L^1(\Om)$. Then $u\in \GSBV (\Om)$ if and only if for all $\xi \in \bS^{n-1}$ there holds $u_y^\xi \in \SBV (\Om_y^\xi)$ for $\LL^{n-1}$\nobreakdash-a.e. $y\in \Om_{\xi}$ and
  \begin{equation*}
    \int_{ \Om_\xi } \bigabs{\D \bigl( (-M) \vee u_y^\xi \wedge M \bigr)} \bigl( \Om_y^\xi \bigr) \dd \LL^{n-1} (y) < \infty \quad \tforall M > 0 \,.
  \end{equation*}
\end{corollary}

\subsection{Convex functions}
\label{convexAnalysis}
Especially, for the numerical part of this paper we also need some theory about convex functions. A good reference for this topic is \cite{HirLem2001} and \cite{EkeTem1999}.

For $\Omega \subset \R^n$ the characteristic function $\chi_\Omega$ over $\Omega$ is given by $\chi_\Omega = 0$ on $\Omega$ and $\chi_\Omega = + \infty$ on $\R^n \setminus \Omega$. It is a convex function if and only if $\Om$ is a convex set. For any function $f \colon \Omega \to \R$, bounded from below by some affine function, $f^\ast \colon \R^n \to \R$ denotes its convex conjugate, i.e.
\begin{equation*}
  f^\ast (y) = \sup_{x\in \R} \bigl( \scprod{x}{y} - f(x)\bigr) \quad \tforall s\in \R^n
\end{equation*}
where $f$ is set to $+\infty$ outside of $\Om$. This definition directly yields Fenchel's inequality, which says
\begin{equation}
  \label{fenchelsInequality}
  \scprod{x}{y} \leq f(x) + f^\ast(y) \quad \tforall x,y \in \R^n \,.
\end{equation}
We remark that $f^\ast$ is always convex and lower semi-continuous and the biconjugate $f^{\ast\ast} = (f^\ast)^\ast$ is the lower semi-continuous convex hull of $f$. Furthermore, $f$ is convex and lower semi-continuous  if and only if $f = f^{\ast \ast}$.

We will also make use of the subdifferential of a function $f\colon \R^n \to (-\infty,+\infty]$, which we denote by $\partial f$. It is given by
\begin{equation*}
	\partial f (x) = \{z \in \R^n  \colon  f(x) - f(y) \leq \scprod{z}{x-y} \text{ for all } y \in \R^n \} \quad \text{for all } x\in\R^n \,.
\end{equation*}
If $f$ is differentiable in $x \in \R^n$, we have $\partial f(x) = \{\nabla f(x) \}$. %

\section{Main Result}
\label{sec:main-result}

For our main result we need several, quite technical assumptions. In order to keep a better overview we first list them here.
\begin{assumption}
  \label{mainAssumptions}
  Let $\eps_0 > 0$. For each $0 < \eps < \eps_0$ let
  \begin{enumerate}[label={[A\arabic*]}]
  \item
    \label{assumptionW}
    $W_\eps \colon [0,1] \to [0, \infty)$ be continuous such that
    $W_\eps \to W$ in $L^1([0,1])$ as $\eps \to 0$ for some $W \in L^1([0,1])$
    with $1 \in \supp W$.
  \item
    \label{assumptionPhi}
    $\varphi_\eps \colon W_\eps ([0,1]) \to \R$ be a convex function such that $\varphi_\eps(W_\eps(1)) \to 0$ and $\varphi_\eps (W_\eps (\cdot)) \to + \infty$ uniformly on $[0,T]$ for all $0 < T <1$, i.e. for all $C>0$ there exists $0 < \tilde \eps < \eps_0$ such that $\varphi_\eps (W_\eps(t)) > C$ for all $t \in [0,T]$ and $\eps < \tilde \eps$.
  \item
    \label{assumptionPsi}
    $\psi_\eps \colon [0,\infty) \to [0,\infty)$ be a convex function such that
    $\lim_{t \to \infty} \frac{\psi_\eps(t)}{t} = c_\eps < \infty$,
    $\psi_\eps (0) \to 0$ and $c_\eps \to c_0 := \int_0^1 W(s) \;\dd s$ as $\eps \to 0$ for $W$ from \ref{assumptionW}, $\varphi^\ast_\eps \leq \psi_\eps$ on $[0,\infty)$, where $\varphi^\ast_\eps$ denotes the convex conjugate of $\varphi_\eps$ (see Section~\ref{convexAnalysis}), and $\psi_\varepsilon(t) \geq  ct + d$ for all $t \geq 0$ and some $c > 0$, $d \in \R$ independent of $\eps$.
  \item
    \label{assumptionEta}
    $\eta_\eps > 0$ such that $\eta_\eps \varphi_\eps(W_\eps(0)) \to 0$ as $\eps \to 0$.
  \end{enumerate}
  Furthermore, assume that
  \begin{enumerate}[resume,label={[A\arabic*]}]
  \item
    \label{assumptionf}
    $f \colon [0,1] \to [0,\infty)$ is a Lipschitz continuous, non-decreasing function with $f(0) = 0$ and $f > 0$ on $(0,1]$.
  \end{enumerate}
\end{assumption}

We are now ready to state our main theorem.
\begin{theorem}
  \label{mainTheorem}
  Let $\Om \subset \R^n$ be a non-empty, open, bounded set with Lipschitz boundary, let $W_\eps$, $\varphi_\eps$, $\psi_\eps$, $\eta_\eps$, $f$  and $c_\eps, c_0>0$ be given as in Assumption~\ref{mainAssumptions}. For each $\eps >0$, we define the functional $F_\eps \colon L^1(\Om) \times L^1 (\Om) \to \R$ by
  \begin{multline}
    \label{ApproxFunc}
    F_\eps (u,v) \coloneq \int_\Om \bigl( f(v) + \eta_\eps \bigr)  \abs{\nabla u}^2 \,\dd x + \int_\Om \varphi_\eps \bigl( W_\eps (v) \bigr) +  \psi_\eps \bigl( \abs{\nabla v} \bigr) \,\dd x \\
    + c_\eps \bigl( \abs{\D^j v}(\Om) + \abs{\D^c v}(\Om) \bigr)
  \end{multline}
  for all $u\in H^1(\Om), v \in \BV(\Om; [0,1]) $ and $F_\eps (u,v) \coloneq +\infty$ otherwise.

  Moreover, define $F \colon L^1(\Om) \times L^1 (\Om) \to \R$ by
  \[
    F (u,v) \coloneq \left\{
      \begin{aligned}
        &\int_\Om f (1) \abs{\nabla u}^2 \,\dd x + 2 c_0 \HH^{n-1} (S_u) && \text{for } u\in \GSBV^2(\Om), v=1 \text{ a.e.,} \\
        &{+ \infty} && \text{otherwise.}
      \end{aligned} \right.
  \]
  Then $F = \Glim_{\eps \to 0} F_\eps$ with respect to the strong topology in $L^1(\Om) \times L^1(\Om)$.
\end{theorem}

For our application in image segmentation we aim for a minimization of $\eqref{SBV_MS}$. In the following corollary we add the missing $L^2$\nobreakdash-penalization term in the functionals~$F$ and $F_\eps$. %
\begin{corollary}
	\label{corollary2}
	Let $\Om \subset \R^n$ be a non-empty, open, bounded set with Lipschitz boundary, let $W_\eps$, $\varphi_\eps$, $\psi_\eps$, $\eta_\eps$, $f$  and $c_\eps, c_0>0$ be given as in Assumption~\ref{mainAssumptions}. Furthermore, let $\beta >0$ and $g\in L^\infty(\Om)$, and let $F_\eps$ and $F$ be given as in Theorem~\ref{mainTheorem}. We define for every $\eps > 0$ the functional
	\begin{equation*}
	G_\eps (u,v) \coloneq
	\left\{
	\begin{aligned}
	&F_\eps(u,v) + \frac\beta2 \int_\Om \abs{u-g}^2 \,\dd x  && \text{for } u\in H^1(\Om), v \in \BV(\Om; [0,1]),\\
	&+ \infty &&\text{otherwise.}
	\end{aligned}
	\right.
	\end{equation*}
	Moreover, we define
	$G \colon L^1(\Om) \times L^1 (\Om) \to \R$ by
	\begin{equation*}
	G (u,v) \coloneq
	\left\{
	\begin{aligned}
	&F(u,v) + \frac\beta2 \int_\Om \abs{u - g}^2  \, \dd x && \text{for } u\in \SBV^2 (\Om) \cap L^\infty (\Om), v=1 \text{ a.e.,} \\
	&+\infty  && \text{otherwise.}
	\end{aligned}
	\right.
	\end{equation*}
	Then, $G = \Glim_{\eps \to 0} G_\eps$ with respect to the strong topology in $L^1(\Om) \times L^1(\Om)$.
\end{corollary}

\begin{proof}
	Since $u \mapsto \int_\Om \abs{u-g}^2 \,\dd x $ is lower semi-continuous, the $\liminf$-inequality follows directly from Theorem~\ref{mainTheorem}. Hence, we have
	\begin{equation}
	\label{eq10}
	G(u,v)\leq \Gliminf_{\eps \to 0} G_\eps(u,v) \quad \text{for all } u,v \in L^1(\Om) \,.
	\end{equation}
	In order to show the $\limsup$-inequality it suffices to consider $u\in \SBV^2(\Om) \cap L^\infty(\Om)$ and $v=1$ a.e., since otherwise, the left hand side of \eqref{eq10} is $+\infty$ and there is nothing to show.

	From Theorem~\ref{mainTheorem} we know that there exists a sequence $(u_\eps,v_\eps)$ in $H^1(\Om) \times \BV(\Om)$ converging to $(u,v)$ as $\eps \to 0$ in the strong $L^1(\Om) \times L^1(\Om)$-topology such that
	\begin{equation*}
	\lim_{\eps \to 0} F_\eps (u_\eps , v_\eps) = F(u,v) \,.
	\end{equation*}
	We consider the truncated function sequence $u_\eps^M $ with $M=\norm{u}_{L^\infty}$, and note that $u_\eps^M \to u$ in $L^2(\Om)$ as $\eps \to 0$. Therefore, we also have
	\begin{equation*}
	  \lim_{\eps \to 0} \int_\Om \abs{u_\eps^M -g}^2 \,\dd x =   \int_\Om \abs{u -g}^2 \,\dd x \,.
	\end{equation*}
	Furthermore, one can easily verify that  $F_\eps (u_\eps ^M,v_\eps ) \leq F_\eps (u_\eps, v_\eps)$, so that
	\begin{equation*}
	\limsup_{\eps \to 0} G_\eps(u_\eps ^M,v_\eps) \leq \limsup_{\eps \to 0} F_\eps(u_\eps,v_\eps) + \frac\beta2  \limsup_{\eps \to 0} \int_\Om \abs{u_\eps^M -g}^2 \, \dd x  = G(u,v) \,,
	\end{equation*}
	which is the required $\limsup$-inequality.
\end{proof}

In view of Proposition~\ref{prop:minconvergence} the existence and compactness of minimizers of the approximating functionals  $G_\eps$ needs to be shown in order to obtain their convergence to a minimizer of the functional $G$.  We give a rigorous proof in the following theorem.
\begin{theorem}
  \label{thm:convergenceOfMinimizers}
	In the setting of Corollary~\ref{corollary2} a minimizer of $G_\eps$ exists for every $\eps >0$.  Furthermore, let $\eps_j$ be an infinitesimal sequence, and let  $(u_{\eps_j} , v_{\eps_j}) \in H^1(\Om) \times \BV(\Om;[0,1])$ be a minimizer of $G_{\eps_j}$ for every $j\in\N$. Then, $v_{\eps_j} \to 1$ in~$L^1(\Om)$, and there exists $u\in \SBV^2(\Om) \cap L^\infty(\Om)$, such that, up to a subsequence, $u_{\eps_j} \to u$ in $L^1 (\Om)$, and $(u,1)$ minimizes $G$.
\end{theorem}
\begin{proof}
	In order to show the existence of minimizers of $G_\eps$ we fix $\eps>0$ and take a minimizing sequence $(u_j,v_j)$ of $G_\eps$, i.e.
	\begin{equation*}
	\lim_{j\to\infty} G_\eps (u_j , v_j) = \inf_{L^1(\Om) \times L^1(\Om)} G_\eps =  \inf_{H^1(\Om) \times \BV (\Om; [0,1])} G_\eps \,.
	\end{equation*}
	In view of $\psi_\eps(t) \geq ct + d$ for all $t \geq 0$ and some $c >0$ as stated in Assumption~\ref{assumptionPsi}, it is easy to see that $(\abs{\D v_j} (\Om))$ is bounded. Further, $(u_j)$ is bounded in $H^1(\Om)$, since $\eta_\eps >0$. By the compactness properties of $\BV(\Om)$ (see \cite[Theorem~3.23]{AmbFusPal2000})  and $H^1 (\Om)$ there exist subsequences of $(u_j)$ and $(v_j)$ (not relabeled) and functions $v\in \BV(\Om)$ and $u\in H^1(\Om)$, such that $v_j \to v$ in $L^1(\Om)$, $\D v_j$ converges sequentially weakly* (in the space of Radon measures) to $\D v$ and $u_j \wto u$ weakly in $H^1 (\Om)$.

	From Fatou's Lemma and \cite[Theorems~5.4 and 5.8]{AmbFusPal2000} we get the lower semi-continuity of $G_\eps$ so that
	\begin{equation*}
	G_\eps (u, v) \leq \liminf_{j\to \infty} G_\eps (u_j ,v_j) \leq \inf_{H^1(\Om) \times \BV(\Om;[0,1])} G_\eps \,.
	\end{equation*}
	Hence, the pair $(u,v)$ minimizes $G_\eps$.

	Now let $(\eps_j)$ be a sequence converging to $0$, and let the pair $(u_{\eps_j}, v_{\eps_j}) $ be a minimizer of $G_{\eps_j}$ for every $j \in \N$. Then we simply have
	\begin{equation*}
		 \int_\Om \varphi_{\eps_j} \bigl( W_{\eps_j} (v_{\varepsilon_j}) \bigr) \,\dd x \leq \min_{L^1(\Om) \times L^1(\Om)} G_{\eps_j} \leq G_{\eps_j} (0,0) = \beta\int_\Om \abs{g}^2 \,\dd x
	\end{equation*}
	which implies, together with Assumption \ref{assumptionPhi}, $v_{\eps_j} \to 1$ in $L^1 (\Om)$ as $\eps_j \to 0$.

	From a simple cut-off argument we get that $\norm{u_{\eps_j}}_{L^\infty (\Om)} \leq \norm{g}_{L^\infty (\Om)}$. Since $f$ is Lipschitz continuous according to Assumption~\ref{assumptionf}, we have $f(v_{\eps_j}) \in \BV(\Omega)$
        with $\D f(v_{\eps_j})$ obeying the chain rule for $\BV$-functions (see \cite[Theorem~3.99]{AmbFusPal2000}). Further, the multiplication operation $(s,t) \mapsto st$ is continuously differentiable and Lipschitz continuous on bounded sets, thus, since both $u_{\eps_j}$ and $f(v_{\eps_j})$ are a.e.~bounded, the product rule for $\BV$-functions holds (see \cite[Theorem~3.99]{AmbFusPal2000}), giving
        $w_{\eps_j}  \coloneq  u_{\eps_j} f(v_{\eps_j}) \in \BV(\Om)$.
        We moreover have
	\begin{multline*}
	\abs{ \D w_{\eps_j} } (\Om) \leq \int_\Om  \abs{\nabla u_{\eps_j} } f(v_{\eps_j}) \,\dd x + \int_\Om  \abs{u_{\eps_j}} f' (v_{\eps_j}) \abs{\nabla v_{\eps_j} } \,\dd x \\
	+ \int_{S_{v_{\eps_j}}} \abs{u_{\eps_j}} \bigl ( f(v_{\eps_j}^+ ) - f(v_{\eps_j}^-) \bigr) \,\dd \HH^{n-1}
        + \int_\Om \abs{u_{\eps_j}} f'(\tilde v_{\eps_j}) \,\dd \abs{\D^c v_{\eps_j}}
        \,.
	\end{multline*}
	Since $u_{\eps_j}$, $f$ and $f'$ are bounded, we can estimate,
        employing a weighted version of the Cauchy--Schwarz inequality and Young's inequality,
	\begin{multline*}
		\abs{ \D w_{\eps_j} } (\Om) \leq C \biggl( 1 + \int_{\Om} f(v_{\eps_j} ) \abs{\nabla u_{\eps_j}}^2 \, \dd x \biggr) \\+ \norm{g}_{L^\infty(\Om)} \norm{ f'}_{L^\infty([0,1])} \biggl( \int_\Om  \abs{\nabla v_{\eps_j}} \,\dd x
		+ \abs{\D^j v_{\eps_j}}(\Om) + \abs{\D^c v_{\eps_j}}(\Om)  \biggr)
	\end{multline*}
	where $C>0$ depends on $f$ and $\Om$.

	By Assumption \ref{assumptionPsi}, $t \leq c^{-1}(\psi_{\varepsilon_j} (t) - d)$ for all $t \geq 0$ and some $c > 0$, $d\in \R$, so
  	\begin{equation*}
		\int_\Om \abs{ \nabla v_{\eps_j}} \,\dd x \leq C \biggl(1+  \int_\Om \psi_{\eps_j} \bigl( \abs{\nabla v_{\eps_j}} \bigr) \,\dd x \biggr)
	\end{equation*}
	with $C>0$ suitably chosen.
        Altogether, since $(c_\varepsilon)$ is bounded, we obtain
	\begin{equation*}
		\abs{ \D w_{\eps_j} } (\Om) \leq C  \bigl(1 + G_{\eps_j}(u_{\eps_j} ,v_{\eps_j}) \bigr) \leq C \biggl( 1 + \beta \int_\Om \abs{g}^2 \,\dd x \biggr) \,,
	\end{equation*}
	where here $C>0$ is a constant depending on $\Om$, $g$, $f$ and $c_0$.  Hence, $\abs{ \D w_{\eps_j} } (\Om)$ is bounded.

	Clearly, $w_{\eps_j}$ is pointwise a.e.~bounded independent of $j$, so by the compactness properties of $\BV(\Om)$ (see \cite[Theorem~3.23]{AmbFusPal2000}) there exists a subsequence of $\eps_j$ (not relabeled) converging to 0, such that $w_{\eps_j}$ converges to some $w$ in $L^1(\Om)$. Since $v_{\eps_j} \to 1$ a.e.~and $f$ is continuous, we also have that $u_{\eps_j} = w_{\eps_j} / f(v_{\eps_j}) \to w / f(1)$ a.e.~as $\eps_j \to 0$. Since $\norm{u_{\eps_j}}_{L^\infty (\Om)}\leq \norm{g}_{L^\infty (\Om)}$, the Dominated Convergence Theorem yields $u_{\eps_j} \to w/f(1)$ in $L^1(\Om)$.

	The assertion now follows from Proposition~\ref{prop:minconvergence} and Corollary~\ref{corollary2}.
\end{proof}

\begin{remark}
	\label{rmrk:choice-of-eta}
	Note that  Theorem~\ref{mainTheorem} and Corollary~\ref{corollary2}  also holds for $\eta_\eps = 0$.  However, for the existence of minimizers of $G_\eps$ we require $\eta_\eps >0$ as indicated in the proof of Theorem~\ref{thm:convergenceOfMinimizers}.
\end{remark}

The following corollary represents a special case of the previous results, which represents the version that is relevant for our numerical computation in Section~\ref{numerics}.
\begin{corollary}
  \label{corollary}
  Let $\Om \subset \R^n$ be a non-empty, open, bounded set with Lipschitz boundary and let $\alpha,\beta, \gamma >0$. For each $\eps > 0$ let $\eta_\eps > 0$ such that $\frac{\eta_\eps}{\eps} \to 0$ as $\eps \to 0$ and define the functionals $G_\eps \colon L^1(\Om) \times L^1 (\Om) \to \R$ by
  \begin{multline*}
    G_\eps (u,v) \coloneq \frac{\alpha}{2} \int_\Om (v^2 + \eta_\eps)
    \abs{\nabla u}^2 \, \dd x + \frac{\beta}{2} \int_\Om \abs{g-u}^2 \,\dd x  \\
    + \frac{\gamma}{2 \eps} \int_\Om (1-v) \,\dd x +
    \frac{\gamma}{2} \abs{\D v}(\Om)
  \end{multline*}
  if $u\in H^1(\Om), v \in \BV(\Om; [0,1])$ and
  $G_\eps (u,v) \coloneq + \infty$ otherwise. Moreover, define
  $G \colon L^1(\Om) \times L^1 (\Om) \to \R$ by
  \begin{equation*}
    G (u,v) \coloneq \frac{\alpha}{2} \int_\Om \abs{\nabla u}^2 \,\dd x + \frac{\beta}{2} \int_\Om \abs{g-u}^2 \,\dd x  + \gamma \HH^{n-1} (S_u)
   \end{equation*}
    for $u\in \SBV^2(\Om) \cap L^\infty (\Om)$, $v=1$  a.e., and $G(u,v) = + \infty$ otherwise.

  Then, for every infinitesimal sequence $(\eps_j)$ a minimizer $(u_{\eps_j} , v_{\eps_j}) $ of $G_{\eps_j}$ exists for every $j\in \N$. Furthermore, $v_{\eps_j} \to 1$ in $L^1(\Om)$, and up to a subsequence $u_{\eps_j} \to u $ in $L^1(\Om)$ with $(u,1)$ being a minimizer of $G$.
\end{corollary}
\begin{proof}[Proof of Corollary~\ref{corollary}]
	We define $\tilde G_\eps \coloneq \frac{2}{\gamma} G_\eps$ and, choose the functions $f$, $W_\eps$, $\varphi_\eps$ and $\psi_\eps$ in the following way:
	\begin{equation*}
	f(t) =  \frac{\alpha}{\gamma} t^2, \qquad W_\eps (t) = (1-t)^\eps, \qquad \varphi_\eps (t) = \frac{1}{\eps} t^\frac{1}{\eps}, \qquad \psi_\eps(s) = s
	\end{equation*}
	for all $t \in [0,1], s\in [0,\infty)$ and $0 < \eps < 1$. Note that in this setting we have
	\begin{equation*}
	\varphi^\ast_\eps (s) =
	\left\{
	\begin{aligned}
	&(1-\eps) (\eps^{2\eps} s)^\frac{1}{1-\eps} && \text{for } s \in [0,\eps^{-2}] \,, \\
	& s - \frac{1}{\eps} && \text{for } s >\eps^{-2} \,,
	\end{aligned}
	\right.
	\end{equation*}
	and hence, one can simply verify that Assumption~\ref{mainAssumptions} is fulfilled with $c_0 = 1$.

	From Theorem~\ref{mainTheorem} we, therefore, get that $\tilde G_\eps$ $\Gamma$-converges to $\tilde G \coloneq \frac{2}{\gamma} G$.
	Since $\Gamma$\nobreakdash{-}convergence is preserved under constant
	multiplication we get the result by multiplying $\tilde G_\eps$ and $\tilde G$
	with $\frac{\gamma}{2}$.
\end{proof}

\section{Proof of Theorem~\ref{mainTheorem}}
\label{proof}
The proof of Theorem~\ref{mainTheorem} follows the usual strategy that has been used for the classical Ambrosio-Tortorelli approximation and various generalizations (see \cite{AmbTor1990,AmbTor1992,Bra1998,DalIur2013,Iur2013,Iur2014}).  Firstly, we show the $\liminf$\nobreakdash-inequality on the real line (see Proposition~\ref{oneDimLiminf}). The generalization to the multi-dimensional case, stated in Proposition~\ref{multDimLiminf}, is then shown by a slicing argument.

The $\limsup$\nobreakdash-inequality is shown with the help of the density result,  Theorem~\ref{densityResult}. Here, we exploit the fact that the phase field variable is allowed to have jumps, which enables the construction of a much simpler recovery sequence than when the phase field needs to be smooth.

\begin{proposition} \label{oneDimLiminf}
  In the setting of Theorem~\ref{mainTheorem} with $\Omega \subset \R$ we redefine $F \colon L^1(\Om) \times L^1 (\Om) \to \R$ by
  \begin{equation*}
    F (u,v) \coloneq
    \left\{
      \begin{aligned}
        & \int_\Om f(1)\abs{u'}^2 \,\dd x + 2 c_0 \# S_u  &  & \text{for } u\in \SBV^2(\Om), v=1 \text{ a.e.}                                  \\
        & {+ \infty}                                      &  & \text{otherwise}
      \end{aligned}
    \right.
  \end{equation*}
  Then there holds $F \leq \Gliminf_{\eps \to 0} F_\eps$.
\end{proposition}
\begin{proof}
  First of all, for each open set $I\subset \Om$ we define the localized functionals
  \begin{multline*}
    F_\eps (u,v;I) \coloneq \int_I \bigl( f(v) + \eta_\eps \bigr) \abs{u'}^2 + \varphi_\eps \bigl( W_\eps (v) \bigr) +  \psi_\eps \bigl( \abs{v'} \bigr) \,\dd x \\
    + c_\eps \bigl( \abs{\D^j v}(I) + \abs{\D^c v}(I) \bigr)
  \end{multline*}
  for all $u \in H^1(I)$ and $v\in \BV(I;[0,1])$, and $F_\eps(u,v;I) \coloneq +\infty$ otherwise.

  Now, let $(\eps_j)$ be a sequence greater than zero with $\eps_j \to 0$ as $j \to \infty$,  and let $(u_j)$ and $(v_j)$ be sequences in $L^1(\Om)$ such that $u_j \to u$ and $v_j \to v$ as $j\to \infty$. By possibly extracting a subsequence, we can assume that
  \begin{equation*}
    \liminf_{j\to\infty} F_{\eps_j} (u_j , v_j) = \lim_{j \to \infty} F_{\eps_j} (u_j,v_j) < \infty \,.
  \end{equation*}
  Therefore, we must have $\int_\Om \varphi_{\eps_j} (W_{\eps_j} (v_j)) \,\dd x < \infty$, and because of to the uniform convergence of $\varphi_{\eps_j} (W_{\eps_j} (\cdot))$ to $ + \infty$ as $\eps\to 0$ (see \ref{assumptionPhi}), we obtain that $v=1$ a.e. on $\Om$.

  We first show that $\# S_u$ is finite and
  \begin{equation} \label{firstStep}
    2 c_0 \# S_u \leq \liminf_{j\to\infty} F_{\eps_j} \bigl( u_j , v_j; B_\delta (S_u) \bigr) \quad \tforall \delta > 0 \text{ sufficiently small}\,.
  \end{equation}
  For that let $y_0 \in S_u$, and let $\delta > 0$ sufficiently small such that $B_\delta (y_0) \subset \Omega$. Set $M \coloneq \liminf_{j\to\infty} \essinf_{B_{\frac{\delta}{2}} (y_0)} (f \circ v_j)$ and assume that $M>0$. Furthermore, let $0 < \kappa < M$ and choose $j_0 > 0$ such that up to a subsequence, there holds $M <  \essinf_{B_{\frac{\delta}{2}} (y_0)} (f\circ v_j) + \kappa$ for all $j > j_0$. Then there holds
  \[
    \int_{y_0 - \frac{\delta}{2}}^{y_0 + \frac{\delta}{2}} \abs{u'_j}^2 \,\dd x \leq \frac{1}{M - \kappa} \int_{y_0 - \frac{\delta}{2}}^{y_0 + \frac{\delta}{2}} f (v_j) \abs{u'_j}^2 \,\dd x \leq \frac{C}{M-\kappa} \quad \tforall j > j_0
  \]
  so that $u_j'$ converges weakly to $u'$ in $L^2 (B_\frac{\delta}{2} (y_0))$ and consequently $y_0 \notin S_u$. Hence, we must have $M=0$, and  we can find a sequence $(y_j)$ such that $f(\tilde v_j (y_j)) \to 0$, where $\tilde v_j$ is a precise representative of $v_j$. The assumptions on $f$ in  \ref{assumptionf} imply $\tilde v_j (y_j) \to 0$ as $j \to \infty$. Since $\tilde v_j \to 1$ a.e.~we can, therefore, find $y^+, y^- \in B_\delta (y_0)$ with $y^- < y_0 < y^+$ such that $\tilde v_j(y^-) \to 1$ as well as $\tilde v_j (y^+) \to 1$.

  With this at hand we get from the $L^1$\nobreakdash-convergence of $W_\eps$ (see \ref{assumptionW}),
  \begin{equation}
    \label{basic}
    2 c_0 = \lim_{j\to\infty}  \biggl[ \int_{\tilde v_j(y_j)}^{\tilde v_j(y^+)} W_{\eps_j} (s) \,\dd s + \int_{\tilde v_j(y_j)}^{\tilde v_j(y^-)} W_{\eps_j} (s) \,\dd s \biggr] \,.
  \end{equation}
  Defining
  \begin{equation}\label{defphi}
    \Phi_{\eps} (t) \coloneq \int_0^t W_\eps (s) \,\dd s \quad \tforall t\in [0,1], \eps > 0
  \end{equation}
  we get, for $j$ large enough,
  \begin{multline*}
    \int_{\tilde v_j(y_j)}^{\tilde v_j(y^+)} W_{\eps_j} (s) \,\dd s + \int_{\tilde v_j(y_j)}^{\tilde v_j(y^-)} W_{\eps_j} (s) \,\dd s \\
    = \bigabs{\Phi_{\eps_j} \bigl(\tilde v_j(y^+) \bigr) - \Phi_{\eps_j} \bigl( \tilde v_j (y_j) \bigr)} + \bigabs{ \Phi_{\eps_j} \bigl( \tilde v_j(y^-) \bigr) - \Phi_{\eps_j} \bigl( \tilde v_j (y_j) \bigr)}
  \end{multline*}
  and together with \eqref{pointwiseVariation}
  \begin{equation}
    \label{applyPointwiseVariation}
    \int_{\tilde v_j(y_j)}^{\tilde v_j(y^+)} W_{\eps_j} (s) \,\dd s + \int_{\tilde v_j(y_j)}^{\tilde v_j(y^-)} W_{\eps_j} (s) \,\dd s \leq \bigabs{ \D (\Phi_{\eps_j} \circ v_j)} \bigl(B_\delta (y_0) \bigr) \,.
  \end{equation}
  Applying Lemma~\ref{lem:tv_compose_est} yields
  \begin{multline}
    \label{applyChainRule}
      \bigabs{ \D (\Phi_{\eps_j} \circ v_j)} \bigl(B_\delta (y_0) \bigr)
      \leq \int_{y_0 - \delta}^{y_0 + \delta}  \varphi_\varepsilon \bigl( W_\varepsilon (v) \bigr) \,\dd{x}   + \int_{y_0 - \delta}^{y_0 + \delta}  \psi_\varepsilon(\abs{v'}) \,\dd{x} \\
      + c_\varepsilon \Bigl( \abs{\D^j v} \bigl(B_\delta (y_0) \bigr) + \abs{\D^c v} \bigl(B_\delta (y_0) \bigr) \Bigr)
  \end{multline}
  By merging \eqref{basic}, \eqref{applyPointwiseVariation} and \eqref{applyChainRule} and since $c_\eps \to c_0$ as $\eps \to 0$ (see \ref{assumptionPsi}) we deduce
  \begin{equation*}
    2c_0 \leq \liminf_{j \to \infty} F_{\eps_j}\bigl( u_j,v_j; B_\delta(y_0) \bigr) \,.
  \end{equation*}

  For every $N\leq \#S_u$ we can repeat the preceding arguments for each element in a set $\{y_1, \dotsc, y_N \} \subset S_u$ with $\delta > 0$ sufficiently small such that $B_\delta(y_k) \cap B_\delta(y_\ell) = \emptyset$ for $k \neq \ell$ in order to obtain
  \begin{equation*}
    2 c_0 N \leq \sum_{k=1}^N \liminf_{j \to \infty} F_{\eps_j} \bigl(u_j,v_j; B_\delta(y_k)\bigr) \leq \liminf_{j \to \infty} F_{\eps_j} \biggl(u_j, v_j; \bigcup_{k=1}^N B_\delta(y_k) \biggr)\,.
  \end{equation*}
  By assumption the right hand side is finite; hence, there must hold $\# S_u < \infty$ and we deduce \eqref{firstStep}.

  In the next step we show that for all $\delta > 0$,
  \begin{equation} \label{secondStep}
    \int_{\Om \setminus B_\delta (S_u) }f(1) \abs{u'}^2 \,\dd x \leq \liminf_{j \to \infty} F_{\eps_j} \bigl(u_j, v_j; \Om \setminus \overline{B_\delta(S_u)}\bigr) \,.
  \end{equation}

  Let $I \coloneq (a,b) \subset \Om$ be an open interval such that $I \cap S_u =\emptyset$.
  For $k\in \N$ and $\ell \in \{1, \dotsc , k\}$ we define the intervals
  \[
    I^k_\ell \coloneq \biggl( a + \frac{\ell-1}{k}(b-a) , a + \frac{\ell}{k} (b-a)\biggr) \,,
  \]
  and we extract a subsequence of $(v_j)$ (not relabeled) such that $\lim_{j \to \infty} \essinf_{I^k_\ell} v_j$ exists for all $\ell$. Moreover, for $0<z<1$ we define the set
  \[
   T^k_z \coloneq \{ \ell \in \{1, \dotsc,k \} \colon \lim_{j\to\infty} \essinf_{I^k_\ell} v_j \leq z \} \,.
  \]
  For every $\ell\in T^k_z$ there exists a sequence $(x_j)$ in $I^k_\ell$ and $y \in I^k_\ell$ such that
  \[
    \lim_{j \to \infty} \tilde v_j (x_j) = \lim_{j\to\infty} \essinf_{I^k_\ell} v_j\qquad  \text{and} \qquad  \tilde v_j(y) \to 1 \,.
  \]
  Thus, analogously to the above it follows that
  \begin{equation*}
    \int_z^1 W (s) \, \dd s\leq \lim_{j \to \infty} \int_{\tilde v_j(x_j)}^{\tilde v_j(y)} W_{\eps_j} (s)\,\dd s \leq \liminf_{j\to\infty} F_{\eps_j} (u_j, v_j ;  I^k_\ell ) \leq C
  \end{equation*}
  for some $C > 0$ by assumption.

  Repeating this argument for every $\ell\in T^k_z$ we get
  \begin{equation*}
    \bigl( \#  T^k_z \bigr) \int_z^1 W(s) \,\dd s \leq \liminf_{j\to\infty} F_{\eps_j} (u_j, v_j ;  I ) \leq C  \,.
  \end{equation*}
  Note that in view of \ref{assumptionW} there holds $\int_z^1 W(s)\,\dd s > 0$, and hence, $\# T_z^k$ is bounded independently of $k$.
  Because $\# T_z^k$ is also non-decreasing with respect to $k$ it remains constant for $k$ large enough. As a consequence, we can pick $\ell^k_1 < \ell_2^k < \dotsb < \ell^k_{N_z} \in T^k_z$ with $N_z \coloneq \max_{k\in \N} \bigl( \#  T^k_z \bigr)$, such that each~$\ell^k_i/k$ converges to some $\theta_i \in [0,1]$ as $k\to \infty$. Define $T_z \coloneq \{y_1,\dotsc, y_{N}\}$ with $y_i \coloneq a+\theta_i (b-a)$. Let $\rho >0$, choose $k > 2(b - a)/\rho$ large enough, and let $\ell \in T_z^k$. Then we have $I^k_{\ell} \subset B_\rho (T_z)$. Therefore,
  \begin{equation}
  \label{eq:10}
  \begin{aligned}
    \liminf_{j\to\infty} f(z) \int_{I \setminus B_\rho (T_z) } \abs{u_j'}^2 \,\dd x
    &\leq \liminf_{j\to\infty} \int_{ I } f (v_j) \abs{u_j'}^2 \,\dd x \\
    &\leq \liminf_{j \to \infty} F_{\eps_j} (u_j, v_j;I) \,.
    \end{aligned}
  \end{equation}
  From \ref{assumptionf} we have $f(z) > 0$, and thus, we obtain  $u_j' \wto u'$ in $L^2(I \setminus B_\rho(T_z))$ up to a subsequence, and consequently $u \in H^1(I \setminus B_\rho(T_z))$. By the  weak lower semi-continuity of the norm we have
  \begin{equation*}
	  f(z) \int_{I \setminus B_\rho (T_z) } \abs{u'}^2 \,\dd x
	  \leq \liminf_{j \to \infty} F_{\eps_j} (u_j, v_j;I) \,.
  \end{equation*}
  Since this inequality holds for all $\rho >0$, we have $u\in H^1(I\setminus T_z)$, and since $u\in \SBV(I) $ with $I \cap S_u = \emptyset$, we deduce that $u\in H^1 (I)$.  Taking the limit for $\rho \to 0$  results in
    \begin{equation*}
  f(z) \int_{I } \abs{u'}^2 \,\dd x
  \leq \liminf_{j \to \infty} F_{\eps_j} (u_j, v_j;I) \,.
  \end{equation*}
  Finally, we take the limit for $z \to 1$ and obtain (note that $f$ is continuous from~\ref{assumptionf})
\begin{equation*}
\int_{I} f(1) \abs{u'}^2 \,\dd x \leq \liminf_{j \to \infty} F_{\eps_j} (u_j, v_j;I) \,.
\end{equation*}
  Since $I \subset \Om$ was chosen arbitrarily such that $I \cap S_u = \emptyset$  we end up with \eqref{secondStep}. Together with \eqref{firstStep} we  obtain
  \begin{equation*}
  	\int_{\Om \setminus B_\delta (S_u) }f(1) \abs{u'}^2  \,\dd x + 2c_0 \#S_u \leq \liminf_{j\to\infty} F_{\eps_j} (u_j,v_j) \,.
  \end{equation*}
  and we conclude the proof by taking the limit for $\delta\to 0$ .
\end{proof}

\begin{proposition} \label{multDimLiminf}
  In the setting of Theorem~\ref{mainTheorem} there holds
  \[
    F(u,v) \leq \Gliminf_{\eps \to 0} F_\eps (u,v) \quad\tforall u,v \in L^1(\Omega)\,.
  \]
\end{proposition}
\begin{proof}
  For the proof we use the usual notation in the setting of slicing, introduced in Section~\ref{functionsOfBoundedVariation}. In what follows let $\xi \in \bS^{n-1}$ and  $y \in \Om_{\xi}$, let $A \subset \Omega$ be open and choose $u,v \in L^1(\Om)$ arbitrarily. We define the localized version of \eqref{ApproxFunc} by
  \begin{multline*}
    F_\eps (u,v;A) \coloneq \int_A \bigl( f(v) + \eta_\eps \bigr) \abs{\nabla u}^2 + \varphi_\eps \bigl( W_\eps (v) \bigr) +  \psi_\eps \bigl( \abs{\nabla v} \bigr) \,\dd x \\
    + c_\eps \bigl( \abs{\D^j v}(A) + \abs{\D^c v}(A) \bigr)
  \end{multline*}
  if $u\in H^1(A), v\in \BV(A;[0,1])$ and $F_\eps (u,v;A) \coloneq {+}\infty$ otherwise. Furthermore, we define for $I\subset \R$ open
  \begin{multline*}
    \overline F_\eps (u,v;I) \coloneq \int_I \bigl( f(v) + \eta_\eps \bigr)  \abs{u'}^2 + \varphi_\eps \bigl( W_\eps (v) \bigr) +  \psi_\eps \bigl( \abs{v'} \bigr) \,\dd x \\
    + c_\eps \bigl( \abs{\D^j v}(I) + \abs{\D^c v}(I) \bigr)
  \end{multline*}
  if $u\in H^1(I), v\in \BV(I;[0,1])$ and $F_\eps (u,v;I) \coloneq {+}\infty$ otherwise.
  We additionally set
  \[
    F^\xi_\eps (u,v; A) \coloneq \int_{ A_\xi} \overline F_\eps \bigl( u_y^\xi, v_y^\xi; A_y^\xi \bigr) \dd\LL^{n-1}(y) \,.
  \]
  From Fubini's theorem and Theorem~\ref{slicingProp} we therefore  obtain
  \begin{multline*}
    F^{\xi}_\eps (u,v;A) =   \int_A \bigl( f(v) + \eta_\eps \bigr) \bigabs{\scprod{\nabla u}{\xi}}^2 + \varphi_\eps \bigl( W_\eps (v) \bigr) +  \psi_\eps \bigl( \abs{\scprod{\nabla v}{\xi}} \bigr) \,\dd x \\
    + c_\eps \bigabs{\scprod{\D^j v}{\xi}}(A) + c_\eps \bigabs{\scprod{\D^c v}{\xi}} (A)
  \end{multline*}
  if $\abs{\scprod{\D u}{\xi}}$ is absolutely continuous with respect to $\LL^n$, and $F^\xi_\eps(u,v ;A) = {+}\infty$ otherwise.
  Thus, there clearly holds
  \begin{equation}\label{sliceineq}
    F^{\xi}_\eps (u,v;A) \leq F_\eps(u,v;A) \,.
  \end{equation}
  From Proposition~\ref{oneDimLiminf} we know that $\overline F (u,v;I) \leq \Gliminf_{\eps \to 0} \overline F_\eps (u,v ;I)$ with
  \begin{equation*}
    \overline F  (u,v; I) \coloneq
    \left\{
      \begin{aligned}
	&\int_I f(1) \abs{u'}^2 \,\dd x + 2 c_0 \# S_u && \text{for } u\in \SBV^2 (I), v=1 \text{ a.e.,} \\
	&{+ \infty} && \text{otherwise.}
      \end{aligned}
    \right.
  \end{equation*}
  Choosing
  \[
    F^\xi (u,v;A) \coloneq \int_{A_\xi} \overline F (u_y^\xi, v_y^\xi; A_y^\xi) \,\dd \LL^{n-1}(y) \,,
  \]
  there holds for all sequences $(u_j)$ and $(v_j)$ with $u_j \to u$ and $v_j \to v$ in $L^1(\Om)$ as $j\to \infty$
  \begin{equation*}
    F^\xi (u,v;A) \leq \int_{A_\xi} \liminf_{j\to \infty} \overline F_\eps \bigl((u_j)_y^\xi, (v_j)_y^\xi; A_y^\xi \bigr) \,\dd\LL^{n-1}(y)\,.
  \end{equation*}
  Fatou's Lemma and \eqref{sliceineq} yield
  \begin{equation} \label{fatouliminf}
    F^\xi (u,v;A) \leq \Gliminf_{\eps \to 0} F^\xi_\eps (u,v; A) \leq \Gliminf_{\eps \to 0} F_\eps (u,v; A)\,.
  \end{equation}

  Moreover, %
  by construction,
  $F^\xi (u,v;A)$ is finite if and only if %
  for a.a.~$y \in A_\xi$ there holds $v_y^\xi =1$ a.e. on $A_y^\xi$, $u_y^\xi \in \SBV^2(A_y^\xi)$ as well as
  \begin{equation*}
    \int_{A_\xi} \int_{A_y^\xi} f(1) \bigabs{(u_y^\xi)'}^2 \,\dd x + 2 c_0 \# S_{u_y^\xi} \,\dd \LL^{n-1} (y) < \infty \,.
  \end{equation*}
  Since there holds for every $M > 0$ and every $u\in L^1(\Om)$ with $u_y^\xi \in \SBV^2(A_y^\xi)$ for a.a. $y\in A_\xi$
  \begin{align*}
    &\int_{ A_\xi } \bigabs{\D \bigl( (-M) \vee u_y^\xi \wedge M \bigr)} \bigl( A_y^\xi \bigr) \dd \LL^{n-1} (y) \\
    \leq{}& \int_{A_\xi} \frac14 \LL^1(A_y^\xi) +
\int_{A_y^\xi} \bigabs{\bigl( (-M) \vee u_y^\xi \wedge M \bigr)'}^2 \,\dd x
   + 2M \# S_{u_y^\xi} \,\dd \LL^{n-1} (y) \\
    \leq{}& \LL^n(A) + C \int_{A_\xi} \int_{A_y^\xi} f(1) \bigabs{\bigl( (-M) \vee u_y^\xi \wedge M \bigr)'}^2 \,\dd x + 2 c_0 \# S_{u_y^\xi} \,\dd \LL^{n-1} (y) \,,
  \end{align*}
  we get by Corollary~\ref{slicingCor} that $F^{\xi} (u,v; A)$ is finite only if $u \in \GSBV^2 (A)$ and $v=1$ a.e.~in $A$. Hence,
  \begin{equation*}
    F^\xi (u,v;A) = \int_A f(1) \abs{\scprod{\nabla u}{\xi}}^2 \,\dd x + 2 c_0 \int_{S_u} \bigabs{\scprod{\nu_u}{\xi}} \,\dd \HH^{n-1}
  \end{equation*}
  if $u \in \GSBV^2(A)$ and $v = 1$ a.e.~in $A$, and $F^\xi (u,v;A)={+}\infty$ otherwise.

  Since $A$ and $\xi$ were chosen arbitrarily, if $v=1$ a.e.~in $A$, then \cite[Theorem~1.16]{Bra1998} and \eqref{fatouliminf} imply
  \begin{align*}
    F(u,v ;A) &= \int_A f(1) \sup_{\xi\in \bS^{n-1}} \abs{\scprod{\nabla u}{\xi}}^2 \,\dd \LL^{n} + 2 c_0 \int_{S_u} \sup_{\xi\in \bS^{n-1}} \bigabs{\scprod{\nu_u}{\xi}} \HH^{n-1}\\
              &\leq \Gliminf_{\eps \to 0} F_\eps (u,v;A) \,. %
  \end{align*}
  Otherwise, the $\liminf$-inequality
  follows directly from~\eqref{fatouliminf} with $\xi$ arbitrary.
\end{proof}

The following proposition now shows the $\limsup$\nobreakdash-inequality.
\begin{proposition} \label{limsupIneq}
  In the setting of Theorem~\ref{mainTheorem} there holds
  \[
    \Glimsup_{\eps \to 0} F_\eps (u,v) \leq F(u,v) \quad\tforall u,v \in L^1(\Omega)\,.
  \]
\end{proposition}
\begin{proof}
  If $u\notin \GSBV^2(\Om)$ or $v \neq 1$ on some set with non-zero measure the assertion is obvious. We first show that the result holds for $u$ replaced by $w \in \SBV^2(\Om) \cap L^\infty (\Om)$ for which \ref{firstDensityProp}--\ref{middleDensityProp} in Theorem~\ref{densityResult} (replacing $w_j$ by $w$) hold.

  For this purpose choose for every $\eps > 0$ some $\delta_\eps>0$ such that $\frac{\eta_\eps}{\delta_\eps} \to 0$ as $\eps \to 0$ but still $\delta_\eps \varphi_\eps(W_{\eps}(0)) \to 0$ as $\eps \to 0$, for instance
  \begin{equation*}
    \delta_\eps = \frac{\sqrt{\eta_\eps}}{\sqrt{\varphi_\eps(W_\eps(0))}} \,.
  \end{equation*}
  Take some smooth cutoff function $\phi \colon \R \to [0,1]$ with $\phi=1$ on $B_{\half}(0)$ and $\phi = 0$ on $\Omega \setminus B_{1}(0)$, and define $\tau(x) = \dist(x, S_w)$ for all $x \in \Om$. Then, we set $\phi_\eps (x)= \phi (\tau(x) / \delta_\eps)$ for all $x\in \Om$, and
  we fix for every $\eps >0$ the function  $w_\eps = (1 - \phi_\eps) w$, for which holds $w_\eps \in H^1(\Om)$, $w_\eps = w$ on $\Om \setminus B_{\delta_\eps} (S_w)$ and $w_\eps \to w$ in $L^1(\Om)$ as $\eps \to 0$. Furthermore we define
  \begin{equation*}
    v_\eps =
      \begin{dcases}%
		0 & \text{on } B_{\delta_\eps} (S_{w}) \cap \Om \,, \\
	1 & \text{elsewhere.}
      \end{dcases}
  \end{equation*}
  Since $\overline{S_w}$ is polyhedral there holds $\HH^{n-1} (\partial B_{\delta_\eps} (S_w) \cap \Om ) < \infty$. Consequently, we have $v_\eps \in \BV(\Om ;[0,1])$ for all $\eps > 0$.

  With this at hand, recalling \ref{assumptionf},  we get
  \begin{multline}
    \label{limsupIneqProof}
      F_\eps (w_\eps, v_\eps)    \leq \int_\Om f(1) \abs{\nabla w}^2 \,\dd x + \eta_\eps \int_\Om \abs{\nabla w_\eps}^2 \, \dd x \\
      + \LL^{n}( \Omega ) \bigl( \varphi_\eps(W_\eps(1)) + \psi_\eps(0) \bigr) +  \LL^{n}\bigl(B_{\delta_\eps} (S_{w}) \bigr) \varphi_\eps(W_\eps(0)) \\
      + \HH^{n-1} \bigl(\partial B_{\delta_\eps} (S_{w}) \bigr) c_\varepsilon  \,.
  \end{multline}
  By the choice of $w_\eps$, the fact that $\norm{w}_{L^\infty(\Om)} \leq M$ and that $\abs{ \nabla \tau(x)} = 1$ a.e. on~$\Om$ (see~\cite[Lemma 3.2.34]{Fed1969}) we get on $B_{\delta_\eps} (S_w)$
  \begin{align*}
    \abs{\nabla w_\eps}
    &\leq \abs{w \nabla \phi_\eps} + \abs{(1-\phi_\eps) \nabla w} \leq \frac{M}{\delta_\eps} \norm{\phi '}_{L^\infty(\Om)} + \abs{\nabla w}\,,
  \end{align*}
  which implies
  \begin{multline*}
    \eta_\eps \int_\Om \abs{\nabla w_\eps}^2 \,\dd x  \leq \eta_\eps \int_{\Om \setminus B_{\delta_\eps}(S_w)} \abs{\nabla w}^2 \,\dd x
    + C \frac{\eta_\eps}{\delta_\eps^2} \LL^n\bigl( B_{\delta_\eps} (S_w) \bigr)\\ + 2 \eta_\eps \int_{B_{\delta_\eps} (S_w)} \abs{\nabla w}^2 \, \dd x
  \end{multline*}
  with $C = 2 M^2 \norm{\phi' }^2_{L^\infty (\Om)}$ independent of $\eps$. The first and the last term obviously converge to 0 as $\eps \to 0$. For the second term we remark that for a polyhedral set, the Hausdorff measure coincides with the Minkowski content (see, e.g., \cite[Theorem~3.2.29]{Fed1969}), so that
  \begin{equation}
    \label{minkowski}
    \frac{\LL^{n}\bigl(B_{\delta_\eps} ( \overline{S_{w}} ) \bigr)}{2\delta_\eps}  \to \HH^{n-1} \bigl( \overline{S_{w}} \bigr)= \HH^{n-1} \bigl( S_{w} \bigr) < \infty \quad  \text{as } \eps \to 0 \,.
  \end{equation}
  As a consequence,  recalling that $\frac{\eta_\eps}{\delta_\eps} \to 0$ we get
  \begin{equation*}
    C \frac{\eta_\eps}{\delta_\eps^2} \LL^n\bigl( B_{\delta_\eps} (S_w) \bigr) \to 0 \quad \text{as } \eps \to 0 \,,
  \end{equation*}
  and therefore
  \begin{equation*}
    \eta_\eps \int_\Om \abs{\nabla w_\eps}^2 \,\dd x \to 0 \quad \text{as } \eps \to 0.
  \end{equation*}
  Additionally, \eqref{minkowski} and $\delta_\eps \varphi_\eps (W_\eps (0)) \to 0$ as $\eps \to 0$  imply
  \begin{equation*}
    \LL^{n}\bigl(B_{\delta_\eps} (S_{w}) \bigr) \varphi_\eps(W_\eps(0)) \to 0 \quad \text{as } \eps \to 0\,.
  \end{equation*}
  Furthermore, there holds
  \begin{equation*}
    \HH^{n-1} \bigl(\partial B_{\delta_\eps} (S_{w}) \bigr) \to 2 \HH^{n-1} (S_{w}) \quad \text{as } \eps \to 0 \,,
  \end{equation*}
  which is again due to $\overline{S_w}$ being a polyhedral set.

  Applying the previous three convergence statements in \eqref{limsupIneqProof} together with the limit behaviour of $\varphi_\eps (W_\eps (1))$, $\psi_\eps(0)$ and $c_\eps$ from \ref{assumptionPhi} and \ref{assumptionPsi}, we get
  \begin{equation} \label{denseLimsupIneq}
    \limsup_{\eps \to 0} F_\eps (w_\eps, v_\eps) \leq F(w, 1) \,.
  \end{equation}

  If $u \in \GSBV^2 (\Omega)$ we have for every $M>0$  that  $u^M \in \SBV^2(\Om) \cap L^\infty( \Om)$ with $u^M \coloneq (-M) \vee u \wedge M$, and we can find a sequence $(w_j)$ in $\SBV^2(\Om) \cap L^\infty (\Om)$ such that \ref{firstDensityProp}--\ref{lastDensityProp}  in Theorem~\ref{densityResult} (replacing $u$ by $u^M$) holds. Together with the lower semi-continuity of $\Glimsup F_\eps$ in $L^1(\Om) \times L^1(\Om)$ and \eqref{denseLimsupIneq} we deduce
  \[
    \Glimsup_{\eps\to 0} F_\eps (u^M, 1) \leq \liminf_{j\to \infty} \Glimsup_{\eps\to 0} F_\eps (w_j,1) \leq \liminf_{j\to \infty} F(w_j, 1) = F(u^M, 1) \,.
  \]

  Obviously, there holds $\norm{\nabla u^M}_{L^2(\Om)} \leq \norm{\nabla u}_{L^2(\Om)} $, and from $S_u = \bigcup_{M>0} S_{u^M}$ (see Section~\ref{functionsOfBoundedVariation}) follows that $\HH^{n-1} (S_{u^M}) \leq \HH^{n-1} (S_u)$. Thus, using again the lower semi-continuity of $\Glimsup F_\eps$ we get
  \begin{equation*}
	  \Glimsup_{\eps\to 0} F_\eps (u, 1) \leq \liminf_{M\to \infty} \Glimsup_{\eps\to 0} F_\eps (u^M,1) \leq \liminf_{M \to \infty} F(u^M, 1) \leq F(u, 1) \,,
  \end{equation*}
   which concludes the proof.
\end{proof}

The proof of Theorem~\ref{mainTheorem} is now a direct consequence of Proposition~\ref{multDimLiminf} and Proposition~\ref{limsupIneq}.

\section{Numerical Examples}
\label{numerics}

The aim of this section is to numerically compare our new approximation from Corollary \ref{corollary} with the classical Ambrosio-Tortorelli approach. We aim for a simple and easy to implement algorithm in order to illustrate the differences between those two models and justify our theory.
As an application for the numerical computations  we choose the image segmentation problem already described in the introduction.

Thus, for $\Om \subset \R^n$ being non-empty, open, bounded and with Lipschitz boundary, we seek to minimize the following functional with respect to $u\in \SBV^2(\Om) \cap L^\infty(\Om)$
\begin{equation}
  \label{Mumford-Shah}
  E(u) =
      \frac{\alpha}{2} \int_\Om \abs{\nabla u}^2 \,\dd x + \frac{\beta}{2} \int_\Om \abs{u - g}^2 \,\dd x + \gamma \HH^1(S_u) \,,
\end{equation}
where $g \in L^\infty (\Om)$ is the original image and $\alpha, \beta, \gamma > 0$ are the parameters influencing the smoothing and segment detection in the solution. They have, of course, to be chosen with care in order to get a sensible result.

Using now Corollary~\ref{corollary} we can approximately minimize $E$ by minimizing
\begin{multline}
  \label{approxMumford-Shah}
  G_\eps (u,v) \coloneq \frac{\alpha}{2} \int_\Om (v^2 + \eta_\eps) \abs{\nabla u}^2 \,\dd x + \frac{\beta}{2} \int_\Om \abs{u - g}^2 \,\dd x \\+ \frac{\gamma}{2\eps} \int_{ \Om} (1-v) \,\dd x + \frac{\gamma}{2} \abs{\D v} (\Om) \,,
\end{multline}
for small $\eps >0$, which we also refer to as the \emph{$\BV$-model}.

On the other hand we consider the elliptic approximation \eqref{AmbrosioTortorelli}, introduced in~\cite{AmbTor1992}:
\begin{multline}
  \label{classicalAmbrosioTortorelli}
  \AT_\eps (u,v) \coloneq \frac{\alpha}{2} \int_\Om (v^2 + \eta_\eps) \abs{\nabla u}^2 \,\dd x + \frac{\beta}{2} \int_\Om \abs{u - g}^2 \,\dd x \\
  + \gamma \int_{ \Om} \frac{1}{4 \epsilon} (1-v)^2  + \eps \abs{\nabla v}^2 \,\dd x
\end{multline}
for $u \in H^1(\Om)$ and $v\in H^1 (\Om; [0,1])$, which we refer to as the \emph{$H^1$-model} (note that we ``redefined'' $\AT_\eps$ as in the following, we will only use \eqref{classicalAmbrosioTortorelli} such that there is no chance of confusion).

For the discretization of these functionals we consider a 2\nobreakdash{-}dimensional image with its natural pixel grid with pixel length $h>0$. If the image is given by $M \times N$ pixels, we use the discrete grid $\Om_h = \{h, \ldots, Mh\} \times \{h, \ldots Nh\}$ and we identify the piecewise constant functions $u, g, v$ as elements in the Euclidean space $\R^{M \times N}$. Precisely, one sets $u = \sum_{ij} u_{ij} \1_{[(i-1)h,ih) \times [(j-1)h,jh)}$ for $(u_{ij}) \in \R^{M \times N}$, where $\1_A$ denotes the characteristic function of $A \subset \R^2$, i.e., $\1_A = 1$ on $A$ and $\1_A = 0$ on~$\R^2 \setminus A$.

For the discretization of the appearing gradients and the total variation we use a finite difference scheme. For this purpose we define the finite difference operator with zero Neumann boundary condition $\nabla_h \colon \R^{M\times N} \to \R^{2\times  M\times N}$ by
\[
(\nabla_h u)_{ij} = \frac{1}{h} \bigl( (\partial_1^+ u)_{ij} , (\partial_2^+ u)_{ij} \bigr)
\quad\tfor u\in \R^{M \times N}
\]
with
\begin{align*}
  (\partial_1^+ u)_{ij} &\coloneq
  \begin{dcases}
   u_{i+1,j} - u_{ij} & \quad\tfor i\in\{ 1,\dotsc,M - 1 \}\,,\\
   0 & \quad \tfor i = M \,,\\
   \end{dcases} \\
  (\partial_2^+ u)_{ij} &\coloneq
  \begin{dcases}
  u_{i,j+1} - u_{ij} &\tfor j\in\{ 1,\dotsc,N - 1 \}\,, \\
  0 & \tfor   j = N\,.
  \end{dcases}
\end{align*}
Furthermore, we denote %
the adjoint of $\nabla_h$ by $-\diver_h$, i.e. for $w \in \R^{2\times M \times N}$ the operator $\diver_h \colon \R^{2\times M \times N} \to \R^{M\times N}$ is defined by
\begin{equation*}
(\diver_h w)_{ij} \coloneq \frac{1}{h} \Bigl( \bigl(  \partial_1^- w^{(1)} \bigr)_{ij} + \bigl( \partial_2^- w^{(2)} \bigr)_{ij} \Bigr) \,,
\end{equation*}
where for all $u\in \R^{M\times N}$
\begin{align*}
(\partial_1^- u)_{ij} &\coloneq
\begin{dcases}
	u_{1j} & \tfor i=1 \,, \\
	u_{ij} - u_{i-1,j} & \tfor i \in \{2,\dotsc, M-1\} \,, \\
	- u_{M-1,j} & \tfor i =M \,,
\end{dcases} \\
(\partial_2^- u)_{ij} &\coloneq
\begin{dcases}
u_{i1} & \tfor j=1 \,, \\
u_{ij} - u_{i,j-1} & \tfor  j\in \{2,\dotsc, N-1\} \,, \\
- u_{i,N-1} & \tfor j =N \,.
\end{dcases}
\end{align*}
For functions $u,v \in \R^{M \times N}$, operations such as the product $uv$ (or $u \cdot v$), the minimum $u \wedge v$, the maximum $u \vee v$, and the square $u^2$ are always meant to be element-wise.
With $\norm{u}_2$, $\norm{u}_1$ and $\norm{u}_\infty$ we respectively refer to the Frobenius norm, the $\ell^1$-norm of $u$ vectorized, and the maximum norm of $u$. The Frobenius inner product of $u$ and $v$ is written as $\scprod{u}{v}$. For any field $q = (q^{(1)}, q^{(2)}) \in \R^{2\times M\times N}$, like $\nabla_h u$ for $u\in \R^{M \times N}$, we denote by $\abs{q}$ the Euclidean norm along the first axis, i.e. $\abs{q} \in \R^{M\times N}$
\begin{equation*}
  \abs{q}_{ij} = \sqrt{\bigl(q^{(1)}_{ij} \bigr)^2 + \bigl(q^{(2)}_{ij} \bigr)^2}.
\end{equation*}
With this strategy we can define the discretized versions of \eqref{approxMumford-Shah} and \eqref{classicalAmbrosioTortorelli}, respectively, for all $u , v \in \R^{M\times N}$ by
\begin{multline*}
  G^h_\eps (u,v) \coloneq \frac{\alpha}{2} \bignorm{v \abs{\nabla_h u }}_2^2 + \frac{\beta}{2} \norm{u -
    g}_2^2 \\
+ \frac{\gamma}{2\eps} \norm{\1 - v}_1 +
  \frac{\gamma}{2} \bignorm{\abs{\nabla_h v}}_1 + \chi_{\{0\leq v \leq \1\}} (v)
\end{multline*}
and
\begin{multline*}
  \AT^h_\eps (u,v) \coloneq \frac{\alpha}{2} \bignorm{v \abs{\nabla_h u}}_2^2 + \frac{\beta}{2} \norm{u - g}_2^2 \\
   + \frac{\gamma}{4 \eps} \norm{\1 - v}_2^2 + \gamma \eps \bignorm{\abs{\nabla_h v}}^2_2 + \chi_{\{0\leq v \leq \1\}} (v) \,.
\end{multline*}
The symbol $\1$ refers to the discretized function that is one almost everywhere.
Note that we neglected the factor~$h^2$ in the functionals since it does not change their minimum. Moreover, we chose $\eta_\eps = 0$ here, because in the discrete setting, the problem of finding a minimizer stays well-posed for this choice.

\begin{remark}
  \label{rmrk:choice-of-eps}
  The choice of the recovery sequence in the proof of Proposition~\ref{limsupIneq} suggests that the width of the detected contours represented by the phase field variable $v$ correlates with the parameter $\eps$. The precise relation between $\eps$ and the width of the phase field is, however, not known. Examining the structure of the approximating functionals, we expect that it depends, in particular, on the trade-off between the two terms $v^2 \norm{\nabla u}^2_\infty$ and $\frac{1}{4\eps}(1-v)$.

   Although, we would like to have the width of the phase field and therefore $\eps$ extremely small, there is a limit of choice depending on the pixel size $h$. To be more precise, choosing $h_\eps > 0$ depending on $\eps$, it is well known that $\AT_\eps^{h_\eps}$ $\Gamma$\nobreakdash{-}converges as $\eps \to 0$, provided that $h_\eps/{\eps} \to 0$ as $\eps \to 0$ (see \cite{Bou1999,BelCos1994}). We believe that a corresponding statement is also true for the considered $\BV$-phase field approximation. A study of this is, however, outside the scope of the present paper.
\end{remark}

The difficulty in finding a minimizer lies in the non-convex, and for $G_\eps^h$ also non-smooth, structure. In previous works an alternating minimization scheme has been commonly used, exploiting the fact that the functionals are convex in each variable separately (see \cite{Bou1999,ArtForMicPer2015,AlmBel2018}). However, in this work we choose a more recent approach, which is the proximal alternating linearized minimization (in short PALM) presented in \cite{BolSabTeb2014}. This algorithm is a form of an alternating gradient descent procedure, for which we do not have to solve any linear equation. This makes the algorithm also faster than the alternating minimization scheme, especially for rather large images. Our experience also showed no significant difference in the results.

For the PALM algorithm one uses the fact that the objective functional can be written as $J(u,v) + K(u) + H(v)$. Then, for some initial value $u^0, v^0 \in \R^{M \times N}$ we set for each $k \in \N$
\begin{align}
  \label{eq:prox-u}
  u^{k} &= \prox_{t_k}^K \bigl( u^{k-1} - t_k \nabla_u J(u^{k-1},v^{k-1}) \bigr) \,, \\
  \label{eq:prox-v}
  v^{k} &= \prox_{s_k}^H \bigl( v^{k-1} - s_k \nabla_v J(u^k, v^{k-1}) \bigr) \,,
\end{align}
where $t_k, s_k > 0$.
By $\prox_t^g$ we denote the proximal operator with step size $t>0$:
\begin{equation*}
  \prox_t^g (w) = \argmin_{u\in \R^{M\times N}} \biggl( \frac{1}{2t} \norm{u - w}_2^2 + g(u) \biggr) \,.
\end{equation*}
For the right choices of the step sizes $t_k$ and $s_k$ above one can show that this scheme converges to a critical point of $J(u,v) + K(u) + H(v)$ as $k\to \infty$ (see \cite[Proposition~3.1]{BolSabTeb2014}). Namely, we need to choose $t_k = \frac{\theta_1}{L_1(v_{k-1})}$ and $s_k = \frac{\theta_2}{L_2(u_k)}$ for some $\theta_1, \theta_2 \in (0,1)$, where $L_1 (v)$ and $L_2(u)$ are Lipschitz constants of $u \mapsto \nabla_u J(u,v)$ and $v \mapsto \nabla_v J(u,v)$, respectively. Unfortunately, %
convergence rates are not known,
so that as a stopping criterion, we are limited to measure the change of the variables in each iteration. We stop the scheme when this change drops under a specified threshold or if a certain number of iterations is reached.

We will now have a closer look on how the algorithm looks like for $G_\eps^h$ and $\AT_\eps^h$ separately.
\subsection*{$\BV$-model}
We write $G_\eps^h (u,v) = J(u,v) + K(u) + H(v)$ with
\begin{equation}
  \label{eq:split-J-G}
  J(u,v) = \frac{\alpha}{2} \bignorm{v \abs{\nabla_h u}}_2^2 \,, \quad K(u) = \frac{\beta}{2} \norm{u-g}_2^2
\end{equation}
and
\begin{equation*}
  H(v) = \frac{\gamma}{2\eps} \norm{\1 - v}_1 + \frac{\gamma}{2} \bignorm{\abs{\nabla_h v}}_1 + \chi_{\{0 \leq v \leq 1\}} (v) \,.
\end{equation*}
We have
\begin{equation*}
  \nabla_u J(u,v) = -\alpha \diver_h\bigl( v^2 \nabla_h u\bigr) \quad\text{and}\quad \nabla_v J(u,v) = \alpha v \abs{\nabla_h u}^2\,.
\end{equation*}
Since the operator norm  of $\nabla_h$ is strictly below  $\frac{\sqrt {8}}{h}$ (see, e.g. \cite{Chambolle2004,Bredies2018}), we can choose for some $\theta \in (0,1)$
\begin{equation}
  \label{eq:step-size}
  t_k = \frac{h^2}{8 \alpha} \quad\text{and}\quad s_k = \frac{\theta}{\alpha \bignorm{\abs{\nabla_h u^k}^2}_\infty} \,,
\end{equation}
such that $t = t_k$ is constant throughout the algorithm.

As a simple computation shows, solving \eqref{eq:prox-u} is then equivalent to
\begin{equation}
  \label{eq:solving-u}
    u^k = \frac{ \bar u^k  + t \beta g}{1 + t \beta} \quad \text{with} \quad \bar u^k = u^{k-1} + t \alpha \diver_h \bigl((v^{k-1})^2 \nabla_h u^{k-1} \bigr) \,.
  \end{equation}
  By completing squares and ignoring constant terms the problem \eqref{eq:prox-v} can be equivalently reformulated to
\begin{equation}
  \label{eq:min-v}
  v^k \in  \argmin_{v\in \R^{M\times N}} \biggl( \frac{1}{2} \Bignorm{v - \bar v^k - \frac{\gamma s_k}{2 \eps} \1}_2^2 + \frac{\gamma s_k}{2}  \bignorm{\abs{\nabla_h v}}_1 + \chi_{\{0 \leq v \leq \1\}} (v) \biggr)
\end{equation}
with $\bar v^k = v^{k-1} - s_k \alpha v^{k-1} \abs{\nabla_h u^{k}}^2$. Since the non-smooth term $\norm{\abs{\nabla_h v}}_1$ is still present, this minimization can not be solved directly. Instead we tackle the problem with the algorithm introduced by A. Chambolle and T. Pock in \cite{ChaPoc2011}, solving the corresponding primal-dual problem.  Therefore, we define for all $v \in \R^{M \times N}$ and $w \in \R^{2 \times M \times N}$ the functions
\begin{equation*}
  P_k (v) = \frac{1}{2} \Bignorm{v - \bar v^k - \frac{\gamma s_k}{2 \eps} \1}_2^2 + \chi_{\{0 \leq v \leq 1\}} (v) \quad\text{and}\quad Q_k(w) = \frac{\gamma s_k}{2 h} \bignorm{\abs{w}}_1\,,
\end{equation*}
such that \eqref{eq:min-v} is equivalent to
\begin{equation}
  \label{primalProblem}
  v^k \in \argmin \bigl\{ P_k (v) + Q_k (\nabla_1 v) : v \in \R^{M \times N} \bigr\} \,.
\end{equation}
Here, $\nabla_1$ is the forward difference operator $\nabla_h$ for $h=1$.

The corresponding primal-dual saddle point problem is given by
\begin{equation}
  \label{saddlePointProblem}
  \min_{p\in \R^{M\times N} } \max_{q \in \R^{2 \times M \times N}} \bigl(\scprod{\nabla_1 p}{q} + P_k(p) - Q_k^\ast(q) \bigr)
\end{equation}
where $Q^\ast_k$ denotes the convex conjugate of $Q_k$, i.e., $Q_k^\ast = \chi_{\{ \norm{\abs{\cdot}}_\infty \leq \frac{\gamma s_k}{2 h} \}}$. Clearly, for any solution $(p, q)$ of \eqref{saddlePointProblem} we have that $v^k = p$ is
a solution of
\eqref{primalProblem}. We solve \eqref{saddlePointProblem} with \cite[Algorithm~1]{ChaPoc2011}. Namely, for $ 0 < \tau^2 \leq \frac{1}{8}$ and for some  $p_k^0 \in \R^{M \times N}$, $q_k^0 \in \R^{2 \times M \times N}$ as well as $\hat{p}_k^0 \coloneq p_k^0$ we define for all $\ell \in \N$
\begin{align}
  \label{qStep}
  q_k^{\ell} &= \prox^{Q_k^\ast}_{\tau} \bigl(q_k^{\ell - 1} + \tau \nabla_1 \hat p_k^{\ell - 1} \bigr) \,, \\
  \label{pStep}
  p_k^{\ell} &= \prox^{P_k}_{\tau} \bigl(p_k^{\ell - 1} + \tau \diver_1 q_k^{\ell} \bigr) \,, \\
  \label{eq:interpol-step}
  \hat p_k^{\ell} &= 2 p_k^{\ell}  - p_k^{\ell - 1} \,.
\end{align}

Then, \cite[Theorem~1]{ChaPoc2011} guarantees the convergence of $(p_k^\ell, q_k^\ell)$ as $\ell \to \infty$ to a solution of \eqref{saddlePointProblem}.
For a stopping criterion of the primal-dual iteration we consider the primal-dual gap which is for $p\in \R^{M\times N}$ and $q \in \R^{2\times M\times N}$ given by
\begin{equation*}
  \G_k(p,q) = %
  P_k(p) + Q_k(\nabla_1 p) + P_k^\ast(\diver_1 q) + Q_k^\ast(q)\,.
\end{equation*}
It vanishes if and only if  $(p,q)$ solves \eqref{saddlePointProblem}. For this reason, we stop iteration \eqref{qStep}--\eqref{eq:interpol-step} if the corresponding primal-dual gap is smaller than a certain tolerance.

We now continue with the precise computations of the primal-dual steps for the $\BV$-phase field approximation. Since $Q_k^\ast$ is the indicator function of a convex set, the update step \eqref{qStep} is the projection of $q_k^{\ell - 1} + \tau \nabla_1 \hat p_k^{\ell - 1}$ onto $\{\norm{\abs{\cdot}}_\infty \leq  \frac{\gamma s_k}{2 h} \}$ (cf. \cite[Section~6.2]{ChaPoc2011}). Thus, we simply get
\begin{equation*}
  q_k^{\ell} = \frac{ \bar q_k^\ell}{\1 \vee \frac{2 h \abs{\bar q_k^\ell}}{\gamma s_k}} \quad \text{with} \quad  \bar q_k^\ell = q_k^{\ell - 1} + \tau \nabla_1 \hat p_k^{\ell - 1} \,.
\end{equation*}
The proximal operator appearing in \eqref{pStep} can be solved directly. Namely, we get
\begin{equation*}
  0 \in \frac{1 + \tau}{\tau} p_k^\ell - \frac{1}{\tau} \bar p_k^\ell - \bar v^k - \frac{\gamma s_k}{2 \eps} \1 + \partial \chi_{\{0\leq p \leq \1\}} (p_k^\ell)
\end{equation*}
with $\bar p_k^\ell = p_k^{\ell - 1} + \tau \diver_1 q_k^{\ell}$, which yields
\begin{equation*}
  p_k^\ell = 0 \vee \biggl( \frac{\bar p_k^\ell + \tau \bar v^k + \tau \frac{\gamma s_k}{2 \eps} \1}{1 +  \tau} \biggr)  \wedge \1 \,.
\end{equation*}
The primal-dual gap for $p_k^\ell$ and $q_k^\ell$ can be computed explicitly. Taking into account that $Q^\ast (q_k^\ell) =0$ and
\begin{equation*}
	P_k^\ast (\diver_1 q_k^\ell) = \bigscprod{(p_k^\ell)'}{\diver_1 q_k^\ell}- \frac{1}{2} \Bignorm{(p_k^\ell)' - \bar v^k - \frac{\gamma s_k}{2 \eps} \1}_2^2
\end{equation*}
with
\begin{equation*}
(p_k^\ell)' = 0 \vee \biggl( \bar v^k + \frac{\gamma s_k}{2\eps} \1 + \diver_1 q_k^\ell \biggr) \wedge \1\,,
\end{equation*}
it is given by
\begin{multline*}
  \G_k (p_k^\ell, q_k^\ell)
  = \frac{\gamma s_k}{2 h} \bignorm{\abs{\nabla_1 p_k^\ell}}_1 + \bigscprod{(p_k^\ell)'}{\diver_1 q_k^\ell} \\
  + \frac{1}{2} \bigl(\norm{p_k^\ell}_2^2 - \norm{(p_k^\ell)'}_2^2 \bigr)
  - \biggscprod{p_k^\ell - (p_k^\ell)'}{\bar v^k + \frac{\gamma s_k}{2 \eps} \1} \,.
\end{multline*}
Summing up all the previous computations for our $\BV$-phase field model, we get Algorithm~\ref{alg:bv-phase-field} in the appendix, which is the numerical scheme as implemented.

\subsection*{$H^1$-model (Ambrosio-Tortorelli)}
For the elliptic approximation  we use $J$ and $K$ as in \eqref{eq:split-J-G} and only redefine $H$ by
\begin{equation*}
  H(v) \coloneq \frac{\gamma}{4 \eps} \norm{\1 - v}_2^2 + \gamma \eps \bignorm{\abs{\nabla_h v}}_2^2 + \chi_{\{0\leq v \leq \1\}} (v)
\end{equation*}
in order to obtain $\AT_\eps^h (u,v) = J(u,v) + K(u) + H(v)$. Clearly, $s_k$ and $t = t_k$ can also be chosen as before in \eqref{eq:step-size}. Hence, \eqref{eq:prox-u} results again in \eqref{eq:solving-u}. The difference of the algorithm compared to the one for the $\BV$-phase field  appears in  \eqref{eq:prox-v}, which is now equivalent to
\begin{equation*}
  v^k \in \argmin_{v \in \R^{M \times N}} \biggl( \frac{1}{2} \biggnorm{v - \frac{2 \eps \bar v^k + \gamma s_k \1}{2\eps + \gamma s_k} }_2^2 + \frac{2 \gamma \eps^2 s_k}{2\eps + \gamma s_k} \bignorm{\abs{\nabla_h v}}_2^2 + \chi_{\{0 \leq v \leq \1 \}} (v) \biggr) \,.
\end{equation*}
Since this problem is sufficiently smooth it could be easily solved directly, by solving a linear system. Nevertheless, for a better comparability and for saving the effort of solving a large linear equation,  we stay as close as possible to %
the algorithm used for the $\BV$-model.
Thus, we use again the primal-dual scheme as in \eqref{qStep}--\eqref{eq:interpol-step}, where this time we need to choose
\begin{equation*}
  P_k(v) = \frac{1}{2} \biggnorm{v - \frac{2 \eps \bar v^k + \gamma s_k \1}{2\eps + \gamma s_k} }_2^2 + \chi_{\{0 \leq v \leq \1 \}} (v)
\end{equation*}
for $v\in \R^{M\times N}$ and
\begin{equation*}
  Q_k(w) = \frac{\mu}{2} \bignorm{\abs{w}}_2^2 \quad \text{with} \quad \mu = \frac{4 \gamma \eps^2 s_k}{h^2 (2\eps + \gamma s_k)}
\end{equation*}
for $w\in \R^{2 \times M \times N}$. Note, that we have
$Q_k^\ast (w) = \frac{1}{2 \mu} \norm{\abs{w}}_2^2$
and thus \eqref{qStep} yields
\begin{equation*}
  q_k^\ell =  \frac{\mu}{\mu + \tau} \bar q_k^\ell \quad\text{with}\quad \bar q_k^\ell = q_k^{\ell-1} + \tau \nabla_1 \hat p_k^{\ell-1} \,,
\end{equation*}
and \eqref{pStep} results in
\begin{equation*}
  p_k^\ell = 0 \vee \biggl( \frac{1}{1+\tau} \bar p_k^\ell + \frac{\tau (2 \eps \bar v^k + \gamma s_k \1)}{(1 + \tau) (2 \eps + \gamma s_k)} \biggr) \wedge \1 \quad \text{with} \quad \bar p_k^\ell = p_k^{\ell - 1} + \tau \diver_1 q_k^{\ell} \,.
\end{equation*}
The primal-dual gap for this approximation is given by
\begin{multline*}
  \G_k (p_k^\ell,q_k^\ell)
  = \frac{\mu}{2} \bignorm{\abs{\nabla_1 p_k^\ell}}^2_2 + \bigscprod{\diver_1 q_k^\ell}{(p_k^\ell)'} + \frac{1}{2 \mu} \bignorm{\abs{q_k^\ell}}^2_2 \\
  + \half \bigl( \norm{p_k^\ell}^2_2 - \bignorm{(p_k^\ell)'}^2_2 \bigr) - \biggscprod{p_k^\ell - (p_k^\ell)'}{\frac{2 \eps \bar v^k + \gamma s_k \1}{2 \eps + \gamma  s_k} }
\end{multline*}
with
\begin{equation*}
  (p_k^\ell)' = 0 \vee \biggl( \frac{2 \eps \bar v^k + \gamma s_k \1}{2 \eps + \gamma s_k} + \diver_1 q_k\biggr) \wedge \1 \,.
\end{equation*}
Altogether, this yields Algorithm~\ref{alg:h1-phase-field} in the appendix, which is the numerical scheme that we use for computations.

\subsection*{Numerical Results}
With the presented algorithms we perform computations for two different images. For all numerical examples we fix the width of the images to $1$.
The pixel size $h$ then depends on the number of pixels and is given by $h = \frac{L}{\text{\emph{number of horizontal pixels}}}$.

\begin{figure}[tbp]
  \centering
  \begin{minipage}{0.48\textwidth}
    \centering
    \includegraphics[width=\textwidth]{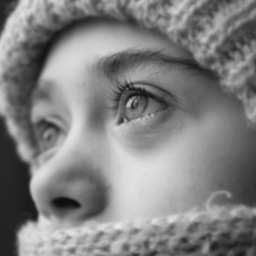}\\
    \footnotesize{original image\footnotemark}
  \end{minipage}%
  \hspace{1em}
  \begin{minipage}{0.48\textwidth}
    \centering
    \includegraphics[width=\textwidth]{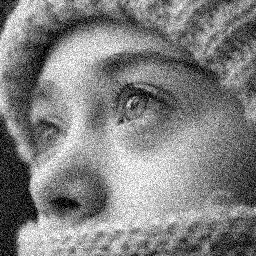}\\
    \footnotesize{noisy image}
  \end{minipage}%
  \caption{Input image with 256 $\times$ 256 pixels for the computations shown in Figure~\ref{fig:face}.}
  \label{fig:face-input-data}
\end{figure}

\begin{table}[tbp]
  \centering
  \setlength{\extrarowheight}{2pt}
  \caption{Numerical parameters}
  \label{tab:parameter}
  \begin{tabular}{ccccccc}
    \hline
    $\alpha$ & $\beta$ & $\gamma$ & $\theta$ & $Tol_1$ & $Tol_2$ & $MaxIt$\\
    \hline
    $1.75 \cdot 10^{-4}$ & $1$ & $3 \cdot 10^{-5}$ & $0.99$  & $10^{-3}$ & $10^{-5}$ & $10000$\\
    \hline
  \end{tabular}
\end{table}

For the first computation we use the noisy image from Figure~\ref{fig:face-input-data}. The latter is generated by adding Gaussian noise of standard deviation $0.1$ and clipping the result to the original image range $[0,1]$. In this computation, the input image $g$ corresponds to this noisy image and we only change the approximating variable $\eps$, in order to investigate its influence, while fixing the other parameters for %
the algorithms
as indicated in Table~\ref{tab:parameter}. The result can be observed in Figure~\ref{fig:face}.

\footnotetext{photo credit: Irina Patrascu Gheorghita: alina's eye \alinaseye ~License: CC-BY 2.0 \license}

\begin{figure}[tbp]
  \centering
  \begin{minipage}{0.48\textwidth}
    \centering
    \setlength{\parskip}{6pt}
    \footnotesize{$\BV$-model}

    \includegraphics[width=0.48\textwidth]{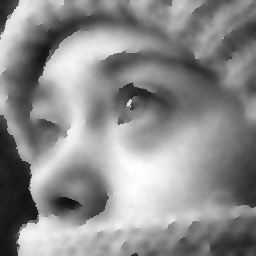}%
    \hfill
    \includegraphics[width=0.48\textwidth]{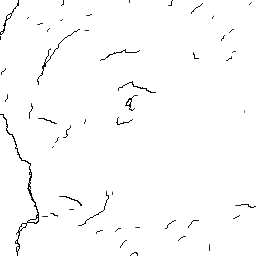}\\
    \footnotesize{$\eps = 5 \cdot 10^{-4}$}

    \includegraphics[width=0.48\textwidth]{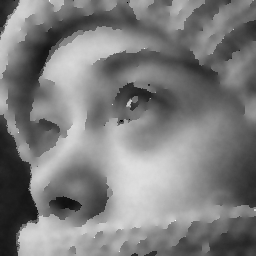}%
    \hfill
    \includegraphics[width=0.48\textwidth]{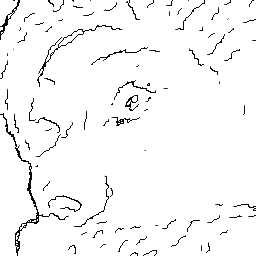}\\
    \footnotesize{$\eps = 1 \cdot 10^{-3}$}

    \includegraphics[width=0.48\textwidth]{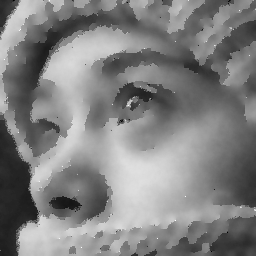}%
    \hfill
    \includegraphics[width=0.48\textwidth]{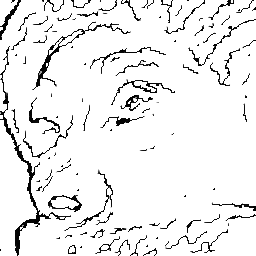}\\
    \footnotesize{$\eps = 2 \cdot 10^{-3}$}

    \includegraphics[width=0.48\textwidth]{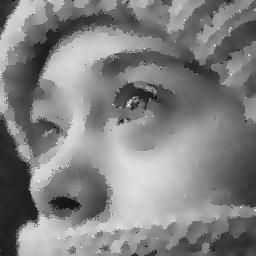}%
    \hfill
    \includegraphics[width=0.48\textwidth]{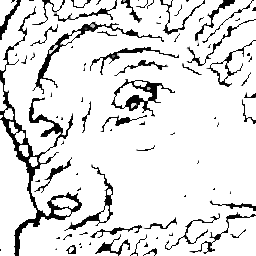}\\
    \footnotesize{$\eps = 3 \cdot 10^{-3}$}

    \includegraphics[width=0.48\textwidth]{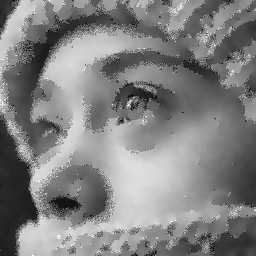}%
    \hfill
    \includegraphics[width=0.48\textwidth]{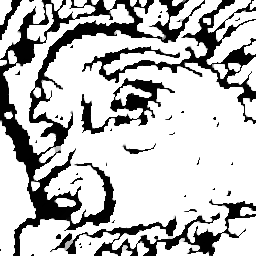}\\
    \footnotesize{$\eps = 5 \cdot 10^{-3}$}
  \end{minipage}%
  \hspace{1em}
  \begin{minipage}{0.48\textwidth}
    \centering
    \setlength{\parskip}{6pt}
    \footnotesize{$H^1$-model}

    \includegraphics[width=0.48\textwidth]{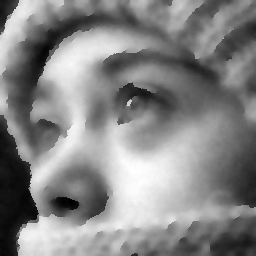}%
    \hfill
    \includegraphics[width=0.48\textwidth]{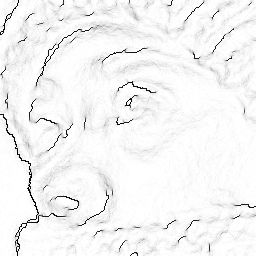}\\
    \footnotesize{$\eps = 2 \cdot 10^{-4}$}

    \includegraphics[width=0.48\textwidth]{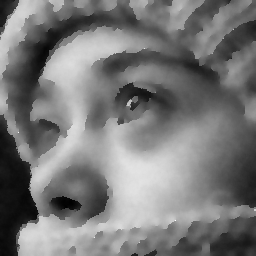}%
    \hfill
    \includegraphics[width=0.48\textwidth]{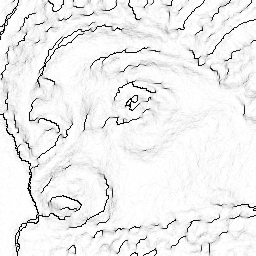}\\
    \footnotesize{$\eps = 3 \cdot 10^{-4}$}

    \includegraphics[width=0.48\textwidth]{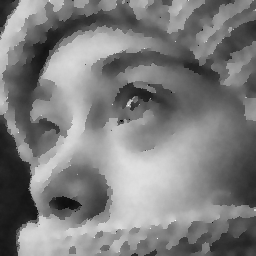}%
    \hfill
    \includegraphics[width=0.48\textwidth]{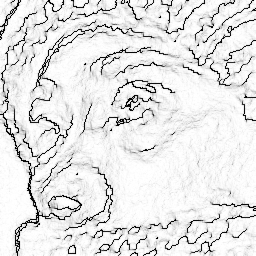}\\
    \footnotesize{$\eps = 5 \cdot 10^{-4}$}

    \includegraphics[width=0.48\textwidth]{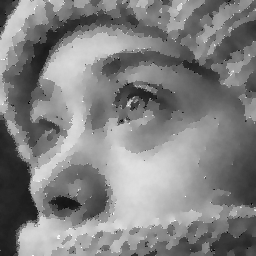}%
    \hfill
    \includegraphics[width=0.48\textwidth]{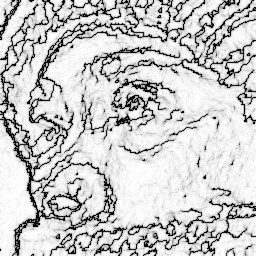}\\
    \footnotesize{$\eps = 1 \cdot 10^{-3}$}

    \includegraphics[width=0.48\textwidth]{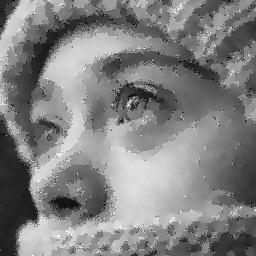}%
    \hfill
    \includegraphics[width=0.48\textwidth]{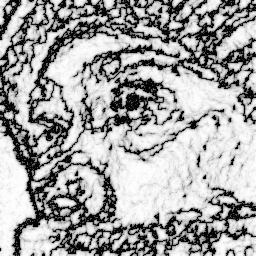}\\
    \footnotesize{$\eps = 1.5 \cdot 10^{-3}$}
  \end{minipage}%
  \caption{Numerical result for different values of $\eps$. The other parameters are given in Table~\ref{tab:parameter}.}
  \label{fig:face}
\end{figure}

One can clearly see that the $\BV$-model produces almost binary phase fields, i.e. $v$ takes only the values $0$ (corresponding to a black pixel) and $1$ (corresponding to a white pixel). In other words these phase fields are much sharper than the ones produced by the $H^1$-model. Moreover, we observe that $\eps$ can be chosen larger when using the $\BV$-model in order to obtain a result that is comparable to the $H^1$-model.

\begin{figure}[tbp]
  \centering
  \begin{minipage}{\textwidth}
    \centering
    \includegraphics[width = 0.48\textwidth]{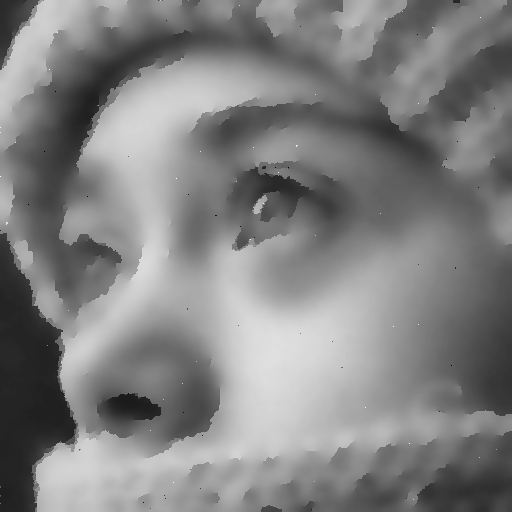}%
    \hspace{1em}
    \includegraphics[width = 0.48\textwidth]{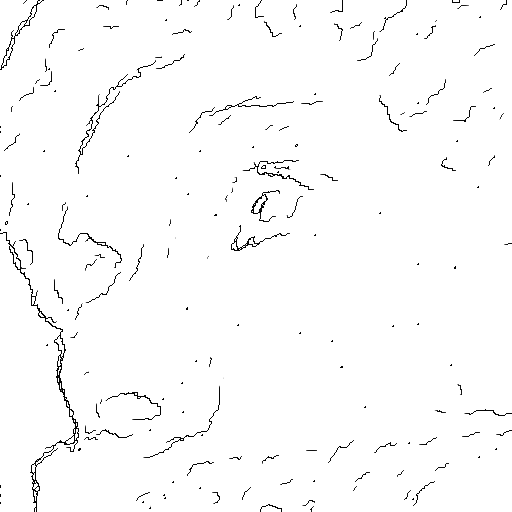}
    \footnotesize{$\BV$-model with $\eps = 5 \cdot 10^{-4}$}
  \end{minipage}\\[1em]
  \begin{minipage}{\textwidth}
    \centering
    \includegraphics[width = 0.48\textwidth]{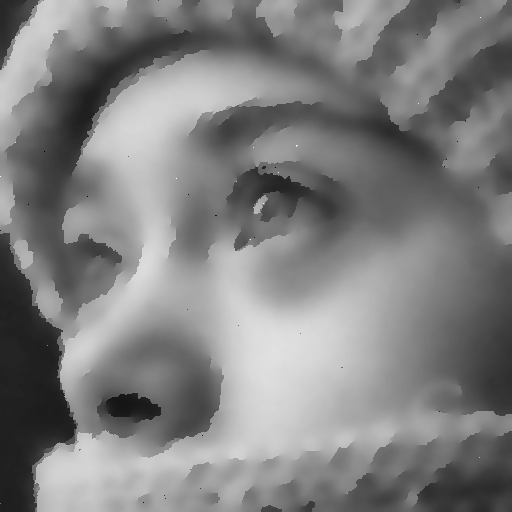}%
    \hspace{1em}
    \includegraphics[width = 0.48\textwidth]{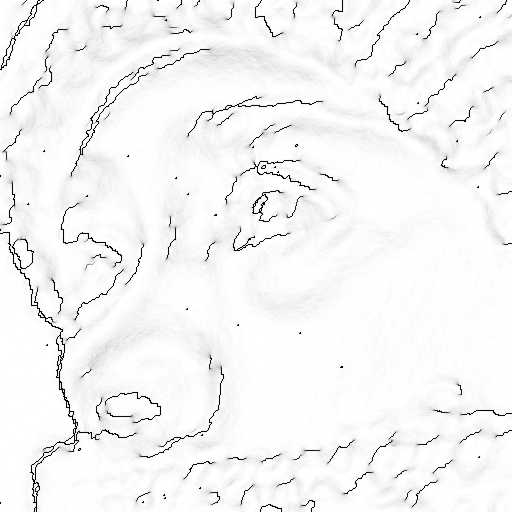}
    \footnotesize{$H^1$-model with $\eps = 2 \cdot 10^{-4}$}
  \end{minipage}%
  \caption{Result of a segmentally denoised image with 512 $\times$ 512 pixels using parameters from Table~\ref{tab:parameter}.}
  \label{fig:large-face}
\end{figure}

Besides the comparison of the two models one can also observe, that in both approximations of the Mumford-Shah functional, only few edges are detected if $\eps$ is too small.  Whereas, if $\eps$ is relatively large, the contours become rather wide. These effects are well-known and have already been mentioned in Remark~\ref{rmrk:choice-of-eps}, from which we also expect that for small values of $\eps$, the phase field may detect the edges again, when reducing $h$. Also this can be confirmed from Figure~\ref{fig:large-face}, where we use the same image but this time with 512 $\times$ 512 pixels keeping the width of the image domain fixed to $1$ as above, resulting in the value of $h$ being halved.

Figure~\ref{fig:sailing} shows another picture with 512 $\times$ 512 pixel size. To the original image we again add Gaussian noise (noise level: $0.1$). This noisy image serves as the input data $g$ for our algorithms. Besides $\alpha$ and $\gamma$, the parameters have a been chosen like in Table~\ref{tab:parameter}.

\afterpage{\clearpage
\begin{figure}[htbp]
  \centering
  \setlength{\parskip}{6pt}
  \begin{minipage}{0.42\textwidth}
    \centering
    \includegraphics[width=\textwidth]{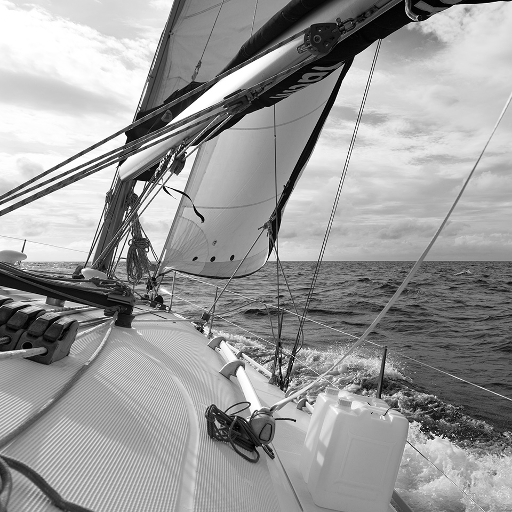}\\
    \footnotesize{original image\footnotemark}
  \end{minipage}%
  \hspace{1em}
  \begin{minipage}{0.42\textwidth}
    \centering
    \includegraphics[width=\textwidth]{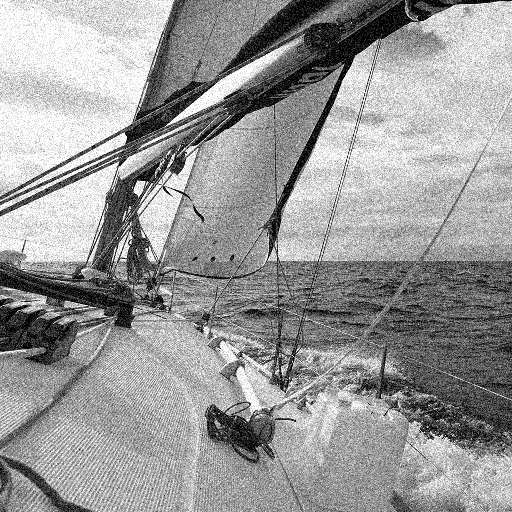}\\
    \footnotesize{noisy image}
  \end{minipage}

  \begin{minipage}{\textwidth}
    \centering
    \setlength{\parskip}{6pt}
    \includegraphics[width=0.42\textwidth]{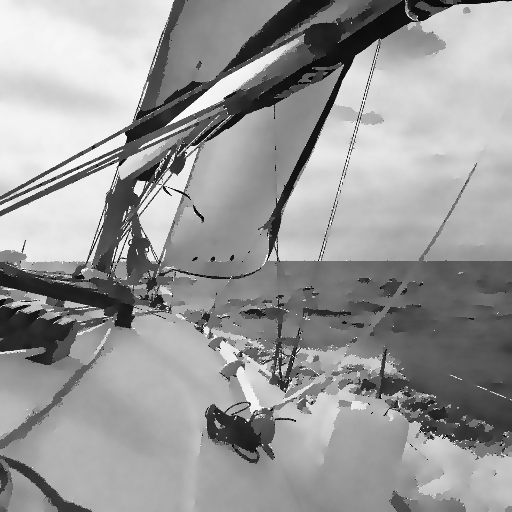}%
    \hspace{1em}
    \includegraphics[width=0.42\textwidth]{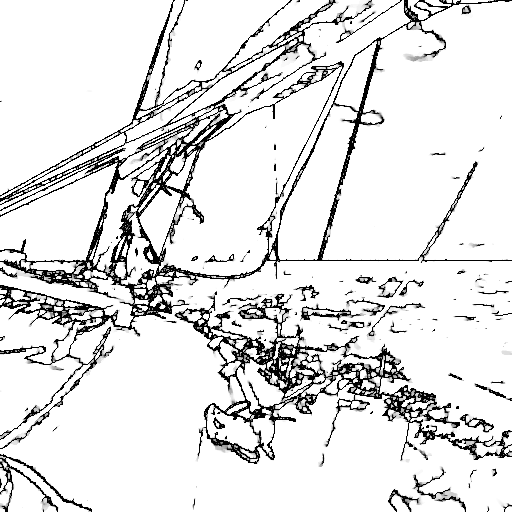}\\
    \footnotesize{$\BV$-model with $\eps = 1 \cdot 10^{-3}$}

    \includegraphics[width=0.42\textwidth]{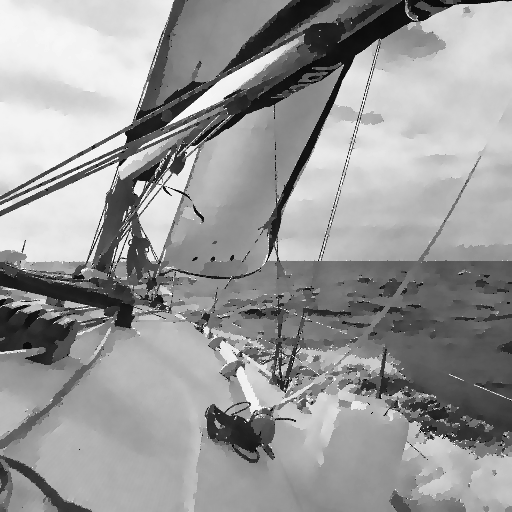}%
    \hspace{1em}
    \includegraphics[width=0.42\textwidth]{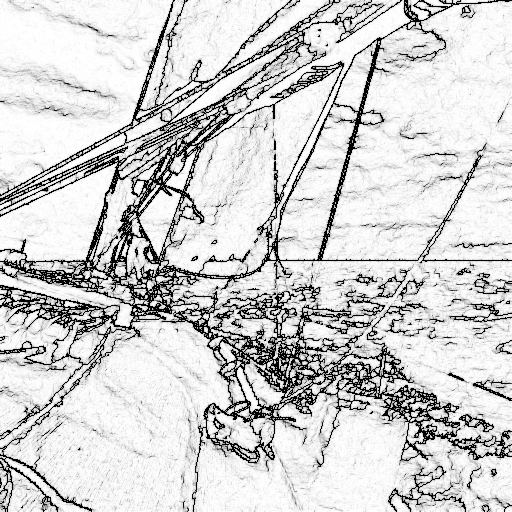}\\
    \footnotesize{$H^1$-model with $\eps = 3\cdot 10^{-4}$}
  \end{minipage}%
  \caption{Image with 512 $\times$ 512 pixels. Computation for $\alpha = 10^{-4}$, $\gamma = 5 \cdot 10^{-6}$ and the other parameters as specified in Table~\ref{tab:parameter}.}
  \label{fig:sailing}
\end{figure}
\footnotetext{photo credit: Phuketian.S: \emph{Sailing from Thailand to Malaysia. Our yacht at the sea} \sailing ~License: CC-BY 2.0 \license}
}

\clearpage
\section*{Acknowledgements}
This work was supported by the International Research Training Group IGDK 1754 ``Optimization and Numerical Analysis for Partial Differential Equations with Nonsmooth Structures'', funded by the German Research Council (DFG) and the Austrian Science Fund (FWF):[W 1244-N18]. 

\bibliographystyle{abbrv}
\bibliography{Bibliography.bib}

\appendix
\section{Auxiliary statements}
\begin{lemma}
  \label{lem:generalized_jensen}
  Let $\mu$ be a signed Radon measure on $\R$,
  $\psi: \R \to [0,\infty]$ a proper, convex and lower semi-continuous function and $\eta \in C_c^\infty(\R; [0,\infty))$ a mollifier, i.e., $\int_\R \eta\,\dd{x} = 1$. Then,
  \[
    \int_\R \psi\bigl( (\mu \ast \eta) \bigr) \,\dd{x}
    \leq \int_{\R} \psi\Bigl(\frac{\dd{\mu}}{\dd{\LL^1}} \Bigr) \,\dd{x} +
      \int_{\R} \psi^\infty\Bigl( \frac{\dd{\mu^s}}{\dd{\abs{\mu^s}}} \Bigr) \,\dd{\abs{\mu^s}}
  \]
  where $\mu = \frac{\dd{\mu}}{\dd{\LL^1}} \LL^1 + \mu^s$ denotes
  the Lebesgue decomposition of $\mu$ and $\psi^\infty(s) = \lim_{t \to \infty} \frac{\psi(st)}{t}$ is the recession function of $\psi$.
\end{lemma}

\begin{proof}
  Fix $x \in \R$, $t > 0$
  and choose $\abs{\mu}_{x,t}
  = \eta(x - \cdot) ( \LL^1 + \frac1t \abs{\mu^s})$ as well
  as $\mu_x = \eta(x - \cdot) \mu$.
  Then, $\abs{\mu}_{x,t}(\R) = 1 + \frac1t \int_\R \eta(x - \cdot) \,\dd{\abs{\mu^s}}$ and
  $\mu_x = \frac{\dd{\mu}}{\dd{\LL^1}} \eta(x - \cdot) \LL^1
  + t \frac{\dd{\mu^s}}{\dd{\abs{\mu^s}}} \frac1t \eta(x -\cdot) \abs{\mu^s}$,
  such that Jensen's inequality yields
  \[
    \begin{aligned}
    &\psi\Bigl(\frac{(\mu \ast \eta)(x)}{\abs{\mu}_{x,t}(\R)} \Bigr) =
    \psi\Bigl( \frac1{\abs{\mu}_{x,t}(\R)} \int_\R \frac{\dd{\mu_x}}{\dd{\abs{\mu}_{x,t}}} \dd{\abs{\mu}_{x,t}} \Bigr) \\
    &\quad \leq
    \frac1{\abs{\mu}_{x,t}(\R)} \Bigl(
    \int_{\R} \psi\Bigl(\frac{\dd{\mu}}{\dd{\LL^1}}(y) \Bigr) \eta(x - y) \,\dd{y} +
      \int_{\R} \frac1t \psi\Bigl( t \frac{\dd{\mu^s}}{\dd{\abs{\mu^s}}}(y) \Bigr) \eta(x - y) \,\dd{\abs{\mu^s}(y)}
    \Bigr).
  \end{aligned}
  \]
  Since $\frac{\dd{\mu^s}}{\dd{\abs{\mu^s}}}$ is either $1$ or $-1$ $\abs{\mu^s}$-almost everywhere, the rightmost integral reads as $\frac1t \psi(t)
  \int_{I_+} \eta(x - \cdot) \,\dd{\abs{\mu^s}}
  + \frac1t \psi(-t) \int_{I_-} \eta(x-\cdot) \,\dd{\abs{\mu^s}}$ where
  $I_+ = \bigl\{x \in \R \colon \frac{\dd{\mu^s}}{\dd{\abs{\mu^s}}}(x) = 1 \bigr\}$ and $I_- = \bigl\{x \in \R \colon \frac{\dd{\mu^s}}{\dd{\abs{\mu^s}}}(x) = -1 \bigr\}$. Clearly, as $t \to \infty$, this expression converges
  to $\int_\R \psi^\infty\bigl( \frac{\dd{\mu^s}}{\dd{\abs{\mu^s}}} \bigr)
  \, \eta(x - \cdot) \,\dd{\abs{\mu^s}}$ (possibly to $\infty$).
  Since $\lim_{t \to \infty} \abs{\mu}_{x,t}(\R) = 1$,
  by lower semi-continuity of $\psi$,
  \[
    \begin{aligned}
      \psi\bigl( (\mu \ast \eta)(x) \bigr)
      &\leq \liminf_{t \to \infty} \psi\Bigl(\frac{(\mu \ast \eta)(x)}{\abs{\mu}_{x,t}(\R)} \Bigr) \\
      &\leq
      \int_{\R} \psi\Bigl(\frac{\dd{\mu}}{\dd{\LL^1}} (y) \Bigr) \eta(x-y) \,\dd{y} +
      \int_{\R} \psi^\infty\Bigl( \frac{\dd{\mu^s}}{\dd{\abs{\mu^s}}}(y) \Bigr)
      \eta(x-y) \,\dd{\abs{\mu^s}}(y).
    \end{aligned}
  \]
  Integrating both sides over $\R$ with respect to $x$ and interchanging
  order on the right-hand side then yields the result.
\end{proof}

\begin{lemma}
  \label{lem:tv_compose_est}
  Let $I := (a,b) \subset \R$ be a bounded open interval,
  $v \in \BV(I; [0,1])$ and $\eps > 0$. Then,
  \[
    \abs{\D(\Phi_\varepsilon \circ v)}(I) \leq
    \int_a^b \varphi_\varepsilon \bigl( W_\varepsilon (v) \bigr) \,\dd{x}
    + \int_a^b \psi_\varepsilon(\abs{v'}) \,\dd{x} + c_\varepsilon
    \bigl( \abs{\D^j v}(I) + \abs{\D^c v}(I) \bigr)
  \]
  with $W_\eps$, $\varphi_\eps$ and $\psi_\eps$
  according to~\ref{assumptionW}, \ref{assumptionPhi} and~\ref{assumptionPsi},
  respectively, and $\Phi_\varepsilon(s) = \int_0^s W_\varepsilon(t) \,\dd t$
  for $s \in [0,1]$.
\end{lemma}

\begin{proof}
  Denote by $v(a) =
  \lim_{\rho \to 0} \frac1\rho \int_a^{a+\rho} v(x) \,\dd{x}$ and $v(b) = \lim_{\rho \to 0} \frac1\rho \int_{b-\rho}^b v(x) \,\dd{x}$ and extend
  $v$ outside of $I$ by $v(x) = v(a)$ for $x \leq a$ and $v(x) = v(b)$ for $x \geq b$. Then, $v \in \BV_{loc}(\R)$ with $\D v$ the zero extension
  of $\D v$ on $I$.
  Choose a mollifier $\eta \in C_c^\infty(\R; [0,\infty))$, $\int_\R \eta \,\dd{x}=1$ and denote by $\eta_\delta(x) = \frac1\delta \eta(\frac{x}{\delta})$ for $\delta > 0$. Then, each $v_\delta = v \ast \eta_\delta$ is in
  $C^\infty(\overline{I}; [0,1])$ and by classical differentiation,
  the Fenchel inequality and \ref{assumptionPsi},
  \[
    \begin{aligned}
      \abs{\D(\Phi_\varepsilon \circ v_\delta)}(I) &=
      \int_a^b W_\varepsilon (v_\delta) \abs{v_\delta'} \,\dd{x}
      \leq \int_a^b \varphi_\varepsilon \bigl( W_\varepsilon(v_\delta) \bigr) +
      \varphi_\varepsilon^*(\abs{v_\delta'}) \,\dd{x} \\
      & \leq \int_a^b \varphi_\varepsilon\bigl(W_\varepsilon(v_\delta) \bigr) \,\dd{x}
      + \int_\R \psi_\varepsilon(\abs{v_\delta'}) \,\dd{x}.
    \end{aligned}
  \]
  We have $v_\delta \to v$ in $L^1(I)$ as $\delta \to 0$, so by
  continuity of $W_\varepsilon$ and $\varphi_\varepsilon$ (as a consequence of
  convexity and finiteness on $W_\varepsilon([0,1])$), one can
  conclude that
  $\int_a^b \varphi_\varepsilon\bigl(W_\varepsilon(v_\delta) \bigr) \,\dd{x} \to
  \int_a^b \varphi_\varepsilon\bigl(W_\varepsilon(v) \bigr) \,\dd{x}$ as $\delta \to 0$. Denoting by $\tilde \psi_\varepsilon(t) = \psi_\varepsilon(\abs{t})$ for
  $t \in \R$ yields a convex function since $\psi_\varepsilon$ is increasing on $[0,\infty)$, so applying Lemma~\ref{lem:generalized_jensen} yields
  \[
    \begin{aligned}
      \int_\R \psi_\varepsilon(\abs{v_\delta'}) \,\dd{x}
      = \int_\R \tilde \psi_\varepsilon\bigl( (\D v \ast \eta_\delta ) \bigr) \,\dd{x}
      &\leq \int_\R \tilde\psi_\varepsilon(v') \,\dd{x} + \int_\R \tilde\psi_\varepsilon^\infty\Bigl(\frac{\dd{\D^s v}}{\dd{\abs{\D^s v}}} \Bigr) \,\dd{\abs{\D^s v}} \\
      &= \int_a^b \psi_\varepsilon(\abs{v'}) \,\dd{x} +
      c_\eps\bigl( \abs{\D^j v}(I) + \abs{\D^c v}(I) \bigr).
    \end{aligned}
  \]
  By continuity
  of $\Phi_\varepsilon$, one can further conclude that $\Phi_\varepsilon \circ v_\delta \to \Phi_\varepsilon \circ v$ in $L^1(I)$ as $\delta \to 0$. The above then implies
  that $\D(\Phi_\varepsilon \circ v_\delta)$ as a sequence of $\delta$
  is bounded in the space of Radon measures on $I$, yielding a weak*-convergent subsequence. By strong-weak*-closedness of the weak derivative, the limit
  has to coincide with $\D(\Phi_\varepsilon \circ v)$. This holds for each subsequence, such that in fact, $\D(\Phi_\varepsilon \circ v_\delta)$ converges weakly* to $\D(\Phi_\varepsilon \circ v)$ as $\delta \to 0$.
  Thus, using weak* lower semi-continuity, we obtain
  \[
    \begin{aligned}
      \abs{\D(\Phi_\varepsilon \circ v)}(I) &\leq
      \liminf_{\delta \to 0} \ \abs{\D(\Phi_\varepsilon \circ v_\delta)}(I) \\
      &\leq \lim_{\delta \to 0} \int_a^b \varphi_\varepsilon\bigl( W_\varepsilon(v_\delta) \bigr) \,\dd{x}
      \int_a^b \psi_\varepsilon(\abs{v'}) \,\dd{x} +
      c_\eps\bigl( \abs{\D^j v}(I) + \abs{\D^c v}(I) \bigr),
    \end{aligned}
  \]
  which, together with the above, yields the desired estimate.
\end{proof}

\clearpage
\section{Pseudo Codes}
\label{sec:pseudo-code}
\begin{algorithm}[htbp]
  \caption{$\BV$-model}
  \label{alg:bv-phase-field}
  \begin{algorithmic}[1]
    \everymath{\displaystyle}
    \setlength{\parskip}{3pt}
    \State $u \leftarrow g$, $v \leftarrow 1$, $q \leftarrow 0$
    \State $t \leftarrow \frac{h^2}{8 \alpha}$, $\tau \leftarrow \frac{1}{\sqrt{8}}$
    \State $it \leftarrow 0$
    \Repeat
      \State $it \leftarrow it + 1$
      \State $u_0 \leftarrow u$, $v_0 \leftarrow v$
      \State $u \leftarrow \frac{u + t \alpha \diver_h( v^2 \nabla_h u) + t\beta g}{1 + t\beta}$
      \State $\displaystyle s \leftarrow \frac{\theta}{\alpha \norm{\abs{\nabla_h u}^2}_\infty}$
      \State $p \leftarrow v$, $ \hat p \leftarrow v$
      \State $\bar v \leftarrow v - s \alpha v \abs{\nabla_h u}^2$
      \Repeat
        \State $p_0 \leftarrow p$
        \State $\bar q \leftarrow q + \tau \nabla_1 \hat p$
        \State $q \leftarrow \frac{\bar q}{1 \vee \frac{2 h}{\gamma s} \abs{\bar q}}$
        \State $\bar p \leftarrow  p + \tau \diver_1 q$
        \State $p \leftarrow 0 \vee \biggl( \frac{\bar p + \tau \bar v + \frac{\gamma \tau s}{2 \eps} \1}{1 +  \tau} \biggr)  \wedge \1$
        \State $p' \leftarrow 0 \vee \biggl( \bar v + \frac{\gamma s}{2\eps} \1 + \diver_1 q \biggr) \wedge \1$
        \State $\hat p \leftarrow 2p - p_0$
        \State $gap \leftarrow  \frac{1}{2} \bigl(\norm{p}_2^2 - \norm{p'}_2^2 \bigr)-  \biggscprod{p - p'}{\bar v + \frac{\gamma s}{2 \eps} \1}$
      \Until{$gap + \frac{\gamma s}{2h} \bignorm{\abs{\nabla_1 p}}_1 + \scprod{\diver_1 q}{p'} \leq Tol_2$}
      \State $v \leftarrow p$
    \Until{$\max\bigl\{\norm{v - v_0}_\infty, \norm{u - u_0}_\infty\bigr\} \leq Tol_1$ \textbf{or} $it = MaxIt$}
  \end{algorithmic}
\end{algorithm}

\begin{algorithm}[htbp]
  \caption{$H^1$-model}
  \label{alg:h1-phase-field}
  \begin{algorithmic}[1]
    \everymath{\displaystyle}
    \setlength{\parskip}{3pt}
    \State $u \leftarrow g$, $v \leftarrow 1$, $q \leftarrow 0$
    \State $t \leftarrow \frac{h^2}{8 \alpha}$, $\tau \leftarrow \frac{1}{\sqrt{8}}$
    \State $it \leftarrow 0$
    \Repeat
      \State $it \leftarrow it +1$
      \State $u_0 \leftarrow u$, $v_0 \leftarrow v$
      \State $u \leftarrow \frac{u + t \alpha \diver_h( v^2 \nabla_h u) + t\beta g}{1 + t\beta}$
      \State $\displaystyle s \leftarrow \frac{\theta}{\alpha \norm{\abs{\nabla_h u}^2}_\infty}$
      \State $p \leftarrow v$, $ \hat p \leftarrow v$
      \State $\bar v \leftarrow v - s \alpha v \abs{\nabla_h u}^2$
      \State $\mu \leftarrow \frac{4 \gamma \eps^2 s}{h^2 (2\eps + \gamma s)}$
      \Repeat
        \State $p_0 \leftarrow p$
        \State $\bar q \leftarrow q + \tau \nabla_1 \hat p$
        \State $q \leftarrow \frac{\mu}{\mu + \tau} \bar q$
        \State $p \leftarrow 0 \vee \biggl( \frac{1}{1+\tau} \bar p + \frac{\tau (2 \eps \bar v + \gamma s \1)}{(1 + \tau) (2 \eps + \gamma s)} \biggr) \wedge \1 $
        \State $p' \leftarrow 0 \vee \biggl( \frac{2 \eps \bar v + \gamma s \1}{2 \eps + \gamma s} + \diver_1 q\biggr) \wedge \1$
        \State $\hat p \leftarrow 2p - p_0$
        \State $gap \leftarrow  \frac{1}{2} \bigl( \norm{p}^2_2 - \norm{p'}^2_2 \bigr) - \biggscprod{p-p'}{\frac{2 \eps \bar v + \gamma s \1}{2 \eps s + \gamma s}}$
      \Until{$gap + \frac{\mu}{2} \bignorm{\abs{\nabla_1 p}}^2_2 + \scprod{\diver_1 q}{p'} + \frac{1}{2 \mu} \bignorm{\abs{q}}^2_2 \leq Tol_2$}
      \State $v \leftarrow p$
    \Until{$\max\bigl\{\norm{v - v_0}_\infty, \norm{u - u_0}_\infty\bigr\} \leq Tol_1$ \textbf{or} $it = MaxIt$}
  \end{algorithmic}
\end{algorithm}

\end{document}